\documentclass[12pt]{amsart} 

\usepackage{amsmath} \usepackage{amssymb} \usepackage{amsthm}
\usepackage{amscd} 
\usepackage[all]{xy} 
\usepackage{enumerate}
\usepackage{mathrsfs}
\usepackage{multirow}
\usepackage{bigdelim}
\usepackage{bm}
\usepackage[bookmarks=false]{hyperref}

\def\tb{\textbf}

\def\ms{\mathscr}
\def\mb{\mathbb}
\def\mf{\mathfrak}

\def\a{{\alpha}}

\def\e{\varepsilon}
\def\E{{\mathcal{E}}} 
\def\F{{\mathcal{F}}} 
\def\G{{\mathcal{G}}}

\def\L{{\mathcal{L}}} 
\def\o{{\omega}}
\def\O{{\mathcal{O}}}

\def\Z{{\mathbb{Z}}}
\def\N{{\mathbb{N}}}
\def\Q{{\mathbb{Q}}} 
\def\R{{\mathbb{R}}}

\def\ol{\overline}
\def\Tr{{\mathrm{Tr}}}

\def\codim{{\mathrm{codim}}}

\def\Spec{{\mathrm{Spec\; }}}

\theoremstyle{plain}
\newtheorem{thm}{Theorem}[section] 
\newtheorem{cor}[thm]{Corollary}

\newtheorem{lem}[thm]{Lemma}
\theoremstyle{definition} 
\newtheorem{defn}[thm]{Definition}

\newtheorem{eg}[thm]{Example} 
\newtheorem{step}{Step} 
\theoremstyle{remark}
\newtheorem{rem}[thm]{Remark}
\newtheorem{obs}[thm]{Observation}
\newtheorem{ques}[thm]{Question}

\newtheorem{notation}[thm]{Notation}
\newtheorem*{cl}{Claim}

\newtheorem*{acknowledgement}{Acknowledgments}

\title{Positivity of anti-canonical divisors and \\ $F$-purity of fibers}
\author{Sho Ejiri}
\address{Graduate School of Mathematical Sciences, the University of Tokyo, 3-8-1 Komaba, Meguro-ku, Tokyo 153-8914, Japan.} 
\email{ejiri@ms.u-tokyo.ac.jp}

\begin{document}
\tolerance = 9999

\maketitle
\markboth{SHO EJIRI}{POSITIVITY OF ANTI-CANONICAL DIVISORS} 
\begin{abstract}
In this paper, we prove that given a flat generically smooth morphism between smooth projective varieties with $F$-pure closed fibers, if the source space is Fano, weak Fano or a variety with the nef anti-canonical divisor, then so is the target space. We also show that relative anti-canonical divisors of generically smooth surjective nonconstant morphisms are not nef and big in arbitrary characteristic.
\end{abstract}
\section{Introduction} \label{section:intro}
Let $f:X\to Y$ be a surjective morphism between smooth projective varieties over an algebraically closed field $k$. Koll\'ar, Miyaoka and Mori \cite[Corollary 2.9]{KMM92} proved that, under the assumption that $f$ is smooth, if $X$ is a Fano variety, that is, $-K_X$ is ample, then so is $Y$. It follows from an analogous argument that, under the same assumption, if $-K_X$ is nef, then so is $-K_Y$ (cf. \cite{Miy93}, \cite[Theorem 1.1]{FG12} and \cite[Corollary 3.15 (a)]{Deb01}). Based on these results, Yasutake asked: ``what positivity condition is passed from $-K_X$ to $-K_Y$?'' Some answers to this question are known in characteristic 0. Fujino and Gongyo \cite[Theorem 1.1]{FG11} proved that, under the assumption that $f$ is smooth, if $X$ is a weak Fano variety, that is, $-K_X$ is nef and big, then so is $Y$. Birkar and Chen \cite[Theorem 1.1]{BC12} showed that, under the same assumption, if $-K_X$ is semi-ample, then so is $-K_Y$. Furthermore, similar but weaker results hold even if $f$ is not smooth (but the characteristic of $k$ is still zero). For example, a result of Prokhorov and Shokurov \cite[Lemma 2.8]{PS09} (cf. \cite[Corollary 3.3]{FG11}) implies that if $-K_X$ is nef and big, then $-K_Y$ is big. Chen and Zhang \cite[Main theorem]{CZ13a} also proved that if $-K_X$ is nef, then $-K_Y$ is pseudo-effective. \par
In contrast, little was known about the positive characteristic case. In this paper, assuming that the geometric generic fiber has only $F$-pure or strongly $F$-regular singularities, we prove that (generalizations of) the statements above hold in positive characteristic, except the one about semi-ampleness. $F$-purity and strong $F$-regularity are mild singularities defined in terms of Frobenius splitting properties (Definition \ref{defn:F-pure}), which have a close connection to log canonical and Kawamata log terminal singularities, respectively. \par
Suppose that $k$ is an algebraically closed field of characteristic $p>0$. 
Let $f:X\to Y$ be a surjective morphism between smooth projective varieties, let $\Delta$ be an effective $\Q$-divisor on $X$ which is $\Q$-Cartier with index $m$, and let $D$ be a $\Q$-divisor on $Y$ which is $\Q$-Cartier with index $n$. Then our main theorem is stated as follows.
\begin{thm}[\textup{Theorem \ref{thm:thmp}}]\label{thm:thmp_intro}\samepage
Let $S$ be a subset of $Y$ such that the following conditions hold for every $y\in S$:
\begin{itemize} 
\item[$($i$)$] $\dim X_y=\dim X-\dim Y$.
\item[$($ii$)$] The support of $\Delta$ does not contain any irreducible component of $X_y$.
\item[$($iii$)$] $(X_{\ol y},\Delta_{\ol y})$ is $F$-pure, where ${\ol y}$ is given by $\Spec \ol{k(y)}$.
\end{itemize}
If $p\nmid m$ and $-(K_X+\Delta+f^*D)$ is nef, then $\O_Y(-n(K_Y+D))$ is weakly positive over an open subset of $Y$ containing $S$.
\end{thm}
Here, weak positivity over $V\subseteq Y$ of line bundles on $Y$ is a property which can be viewed as a generalization of nefness (Definition \ref{defn:positivity}), and is equivalent to nefness if $V=Y$ (Remark \ref{rem:wp}). In Section \ref{section:thmp}, Theorem \ref{thm:thmp} is proved in a more general setting. \par
The following two theorems are corollaries of Theorem \ref{thm:thmp_intro}.
\begin{thm}[\textup{Corollary \ref{cor:gl}}] \label{thm:gl_intro} \samepage
Assume that $f$ is flat, the support of $\Delta$ does not contain any component of any fiber, and that $(X_{y},\Delta_{y})$ is $F$-pure for every closed point $y\in Y$. 
\begin{itemize}
\item[$(1)$] If $p\nmid m$ and $-(K_X+\Delta+f^*D)$ is nef, then so is $-(K_Y+D)$.
\item[$(2)$] If $-(K_X+\Delta+f^*D)$ is ample, then so is $-(K_Y+D)$.
\end{itemize}
\end{thm}
\begin{thm}[\textup{Corollary \ref{cor:gen}}] \label{thm:gen_intro} \samepage
Assume that $(X_{\ol\eta},\Delta_{\ol\eta})$ is $F$-pure, where $\ol\eta$ is the geometric generic point of $Y$. 
\begin{itemize} \item[$(1)$] If $p\nmid m$ and $-(K_X+\Delta+f^*D)$ is nef, then $-(K_Y+D)$ is pseudo-effective. 
\item[$(2)$] If $-(K_X+\Delta+f^*D)$ is ample, then $-(K_Y+D)$ is big. 
\item[$(3)$] If $(X_{\ol\eta},\Delta_{\ol\eta})$ is strongly $F$-regular and $-(K_X+\Delta+f^*D)$ is nef and big, then $-(K_Y+D)$ is big.  \end{itemize}
\end{thm}
Theorem \ref{thm:gl_intro} is a generalization of \cite[Corollary 2.9]{KMM92} and \cite[Corollary 3.15]{Deb01} in positive characteristic.  We can also recover \cite[Corollary 2.9]{KMM92} in characteristic zero from Theorem \ref{thm:gl_intro} by standard reduction to characteristic $p$ techniques. Our proof relies on a study of the positivity of direct image sheaves for $f$ in terms of the Grothendieck trace of the relative Frobenius morphism. This is completely different from the proof of Koll\'ar, Miyaoka and Mori which relies on a detailed study of rational curves on varieties. Theorem \ref{thm:gen_intro} should be compared with \cite[Lemma~2.8]{PS09} and \cite[Main Theorem]{CZ13a}. \par
The following two theorems are direct consequences of Theorems \ref{thm:gl_intro} and \ref{thm:gen_intro}. 
\begin{thm}[\textup{Corollary \ref{cor:cbf}}]\label{thm:cbf_intro}\samepage
Assume that $(X_{\ol\eta},\Delta_{\ol\eta})$ is $F$-pure, where $\ol\eta$ is the geometric generic point of $Y$. If $p\nmid m$ and $K_X+\Delta$ is numerically equivalent to $f^*(K_Y+L)$ for some $\Q$-divisor $L$ on $Y$, then $L$ is pseudo-effective.
\end{thm}
\begin{thm}[\textup{Corollary \ref{cor:fano}}] \label{thm:fano_intro}\samepage 
Assume that $f$ is flat, that every closed fiber is $F$-pure, and that the geometric generic fiber is strongly $F$-regular. If $X$ is a weak Fano variety, that is, $-K_X$ is nef and big, then so is $Y$. 
\end{thm}
Theorem \ref{thm:fano_intro} is a positive characteristic counterpart of \cite[Thorem 1.1]{FG11}. \par
For another application of Theorem \ref{thm:thmp_intro}, we return to the situation where $k$ is of arbitrary characteristic. Suppose that $f:X\to Y$ is a generically smooth surjective morphism between smooth projective varieties over an algebraically closed field of arbitrary characteristic and in addition that the dimension of $Y$ is positive.
\begin{thm}[\textup{Corollary \ref{cor:notample} and Theorem \ref{thm:notample0}}]\label{thm:notample_intro}\samepage
$-K_{X/Y}$ is not nef and big.
\end{thm}
\begin{thm}[\textup{Corollary \ref{cor:notbig} and Theorem \ref{thm:notbig0}}]\label{thm:notbig_intro}\samepage
Assume that $\o_{X_{\ol\eta}}^{-m}$ is globally generated for an integer $m>0$, where $\ol\eta$ be the geometric generic point of $Y$. Then $f_*\o_{X/Y}^{-m}$ is not big in the sense of Definition \ref{defn:positivity}.
\end{thm}
In both of the theorems, the characteristic zero case is proved by reduction to positive characteristic. Theorem \ref{thm:notample_intro} improves a result of Koll\'ar, Miyaoka and Mori \cite[Corollary 2.8]{KMM92} which states that $-K_{X/Y}$ is not ample. Theorem \ref{thm:notbig_intro} includes a result of Miyaoka \cite[COROLLARY 2']{Miy93} which states that if $\o_{X/Y}^{-1}$ is $f$-ample and $\o_{X/Y}^{-m}$ is $f$-free for an integer $m>0$, then $f_*\o_{X/Y}^{-m}$ is not an ample vector bundle. \par
\subsection{Notation} \label{subsection:notat}
Let $k$ be a field.
A {\it $k$-scheme} is a separated scheme of finite type over $k$. 
A {\it variety} means an integral $k$-scheme. 
Let $\varphi:S\to T$ be a morphism of $k$-schemes and let $T'$ be a $T$-scheme. 
Then we denote by $S_{T'}$ and $\varphi_{T'}:S_{T'}\to T'$ respectively 
the fiber product $S\times_{T}T'$ and its second projection. 
For a Cartier or $\Q$-Cartier divisor $D$ on $S$ (resp. an $\O_S$-module $\G$), 
the pullback of $D$ (resp. $\G$) to $S_{T'}$ is denoted by $D_{T'}$ (resp. $\G_{T'}$) if it is well-defined. 
Similarly, for a homomorphism of $\O_S$-modules $\alpha:\F\to\G$, 
the pullback of $\alpha$ to $S_{T'}$ is denoted by $\a_{T'}:\F_{T'}\to \G_{T'}$. 
Assume that the characteristic of $k$ is positive.
$k$ is said to be \textit{$F$-finite} if the field extension $k/k^p$ is finite.
For a $k$-scheme $X$, $F_X:X\to X$ is the absolute Frobenius morphism. 
We often denote the source of $F_X^e$ by $X^{e}$. 
Let $f:X\to Y$ be a morphism between schemes of positive characteristic. 
The same morphism is denoted by $f^{(e)}:X^e\to Y^e$ when we regard $X$ (resp. $Y$) as $X^e$ (resp. $Y^e$). 
We define the $e$-th relative Frobenius morphism of $f$ to be 
the morphism $F^{(e)}_{X/Y}:=(F_X^e,f^{(e)}):X^e\to X\times_Y Y^e=:X_{Y^e}$. 
For a prime $p\in \Z$, $\Z_{(p)}$ denotes the localization of $\Z$ at $(p)=p\Z$.

\begin{small}
\begin{acknowledgement}
The author is greatly indebted to his supervisor Professor Shunsuke Takagi 
for suggesting the problems in this paper and for much valuable advice.
He wishes to express his gratitude to Professor Yoshinori Gongyo 
for fruitful conversations and for showing him the proof of Corollary \ref{cor:fano} (2).
He is grateful to Professors Zsolt Patakfalvi and Hiromu Tanaka 
for suggesting improvements of results in this paper.
He would like to thank Professors Caucher Birkar, Yifei Chen and Osamu Fujino for stimulating discussions and helpful comments.
He also would like to thank Professor Akiyoshi Sannai, 
and Doctors Takeru Fukuoka, Akihiro Kanemitsu, Eleonora Anna Romano, Kenta Sato, Antoine Song and Fumiaki Suzuki for useful comments. 
He was supported by JSPS KAKENHI Grant Number 15J09117 and the Program for Leading Graduate Schools, MEXT, Japan.
\end{acknowledgement}
\end{small}
\section{Trace maps of relative Frobenius morphisms} \label{section:inv}
In this section, we work over an $F$-finite field $k$. 
Let $f:X\to Y$ be a morphism from a pure dimensional Gorenstein $k$-scheme $X$ to a regular variety $Y$. 
Let $K_X$ be a Cartier divisor such that $\O_X(K_X)$ is isomorphic to the dualizing sheaf $\o_X$ of $X$. 
Set $K_{X/Y}:=K_X-f^*K_Y$. For each $d,e>0$ we define some notation by the following commutative diagram.
$$ \xymatrix@R=25pt @C=35pt{ X^{de} \ar[d] \ar@/^25pt/[dd]^(0.5){F_{X^{di}/Y^{di}}^{(d(e-i))}} \ar@/_30pt/[dddd]|(0.7){F_{X/Y}^{(de)}} \ar@/_50pt/[ddddd]_(0.8){f^{(de)}} &&&&&&\\ \vdots \ar[d] &&& \ddots \ar[dr]^{F_X^d} & & & \\ X^{di}_{Y^{de}} \ar[d] \ar@/^50pt/[dd]|(0.6){F_{X_{Y^{d(e-i)}}/Y^{d(e-i)}}^{(di)}} \ar@/^100pt/[ddd]^(0.7){f_{Y^{de}}^{(di)}} && & \cdots \ar[r] & X^{2d} \ar[dr]^{F_X^d} \ar[d]_{F_{X^d/Y^d}^{(d)}} & & \\ \vdots \ar[d] && & \cdots \ar[r] & X^d_{Y^{2d}} \ar[r] \ar[d]_{F_{X_{Y^d}/Y^d}^{(d)}} & X^d \ar[dr]^{F_X^d} \ar[d]_{F_{X/Y}^{(d)}} \\ X_{Y^{de}} \ar[d] && & \cdots \ar[r] & X_{Y^{2d}} \ar[r] \ar[d]_{f_{Y^{2d}}} & X_{Y^d} \ar[r]_{(F_Y^d)_X} \ar[d]_{f_{Y^d}} & X \ar[d]_f \\ Y^{de} && & \cdots \ar[r]_{F_Y^d} & Y^{2d}\ar[r]_{F_Y^d} & Y^d \ar[r]_{F_Y^d} & Y \\ } $$
Since $F_Y$ is flat, every horizontal morphism in the diagram is a Gorenstein morphism, 
and thus every object in the diagram is a pure dimensional Gorenstein $k$-scheme. 
We denote by $\Tr_{F^{(1)}_{X/Y}}:{F^{(1)}_{X/Y}}_*\o_{X^1}\to\o_{X_{Y^1}}$ 
the morphism obtained by applying the functor 
$\ms Hom_{\O_{X_{Y^1}}}(\underline{\quad},\o_{X_{Y^1}})$ 
to the natural morphism 
${F^{(1)}_{X/Y}}^{\#}:\O_{X_{Y^1}}\to{F^{(1)}_{X/Y}}_*\O_{X^1}$. 
For each $e>0$ we define 
\begin{align*} \phi^{(1)}_{X/Y}&:=\Tr_{F^{(1)}_{X/Y}}\otimes\O_{X_{Y^1}}(-K_{X_{Y^1}}):{F^{(1)}_{X/Y}}_*\O_{X^1}((1-p)K_{X^1/Y^1})\to\O_{X_{Y^1}},\quad\textup{and} \\ \phi^{(e+1)}_{X/Y}&:=\left(\phi^{(e)}_{X/Y}\right)_{Y^{e+1}}\circ {F^{(e)}_{X_{Y^1}/Y^1}}_*\left(\phi^{(1)}_{X^{e}/Y^{e}}\otimes\O_{X^e_{Y^{e+1}}}((1-p^{e})K_{X^{e}_{Y^{e+1}}/Y^{e+1}})\right)\\ & \qquad: {F^{(e+1)}_{X/Y}}_*\O_X((1-p^{e+1})K_{X^{e+1}/Y^{e+1}})\to\O_{X_{Y^{e+1}}}. \end{align*} \indent 
Let $E$ be an effective Cartier divisor on $X$, 
let $a>0$ be an integer not divisible by $p$, 
and let $d>0$ be the smallest integer satisfying $a|(p^d-1)$. 
For each $e>0$ we define 
\begin{align*} \L^{(de)}_{(X/Y,E/a)}&:=\O_{X^{de}}((1-p^{de})a^{-1}(aK_{X^{de}/Y^{de}}+E))\subseteq\O_{X^{de}}((1-p^{de})K_{X^{de}/Y^{de}}), \\ \phi^{(d)}_{(X/Y,E/a)}&:{F^{(d)}_{X/Y}}_*\L^{(d)}_{(X/Y,E/a)}\to{F^{(d)}_{X/Y}}_*\O_{X^d}((1-p^d)K_{X^{d}/Y^{d}})\xrightarrow{\phi^{(d)}_{X/Y}}\O_{X_{Y^d}},\quad\textup{and}\\ \phi^{(d(e+1))}_{(X/Y,E/a)}&:=\left(\phi^{(de)}_{(X/Y,E/a)}\right)_{Y^{de}}\circ{F^{(de)}_{X_{Y^{d}}/Y^{d}}}_*\left(\phi^{(d)}_{(X^{de}/Y^{de},E^{de}/a)}\otimes(\L^{(de)}_{(X/Y,E/a)})_{Y^{d(e+1)}}\right) \\ &:{F^{(d(e+1))}_{X/Y}}_*\L^{(d(e+1))}_{(X/Y,E/a)}\to \O_{X_{Y^{d(e+1)}}}.\end{align*} \indent
In order to generalize the definitions above, we recall the generalized divisors on $k$-schemes.
Let $X$ be a $k$-scheme of pure dimension satisfying $S_2$ and $G_1$. 
An \textit{AC divisor} (or \textit{almost Cartier divisor}) on $X$ is 
a reflexive coherent $\O_X$-submodule of the sheaf of total quotient rings of $X$ 
which is invertible in codimension one (see \cite{Har94,MS12}). 
For an AC divisor $D$ we denote the coherent sheaf defining $D$ by $\O_X(D)$. 
The set of AC divisors $WSh(X)$ has a structure of additive group \cite[Corollary 2.6]{Har94}. 
A \textit{$\Z_{(p)}$-AC divisor} is an element of $WSh(X)\otimes_\Z \Z_{(p)}$. 
An AC divisor $D$ is said to be \textit{effective} if $\O_X\subseteq\O_X(D)$, 
and a $\Z_{(p)}$-AC divisor $\Delta$ is said to be \textit{effective} if $\Delta=D\otimes r$ 
for an effective AC divisor $D$ and some $0\le r\in\Z_{(p)}$. \par
Let $f:X\to Y$ be a morphism from a pure dimensional $k$-scheme $X$ to a regular variety $Y$. 
Assume that $X$ satisfies $S_2$ and $G_1$. 
Let $E$ be an effective AC divisor on $X$,
let $a>0$ be an integer not divisible by $p$, 
and let $d>0$ be the smallest integer satisfying $a|(p^d-1)$. 
Then for each $e>0$ we define 
\begin{align*}
\L^{(de)}_{(X/Y,E/a)}&:={\iota_{Y^{de}}}_*\L^{(de)}_{(U/Y,E|_U/a)}\textup{, and}\\ 
\phi^{(de)}_{(X/Y,E/a)}&:={\iota_{Y^{de}}}_*(\phi^{(de)}_{(U/Y,E|_U/a)}):
{F^{(de)}_{X/Y}}_*\L^{(de)}_{(X/Y,E/a)}\to\O_{X_{Y^{de}}}, 
\end{align*} 
where $\iota:U\hookrightarrow X$ is a Gorenstein open subset of $X$ 
such that $\codim_X(X\setminus U)\ge2$ and that $E|_U$ is a Cartier divisor on $U$. 
If $X$ is normal, then for each $e>0$ such that $E:=(p^e-1)\Delta$ is integral, 
we denote $\L^{(e)}_{(X/Y,E/p^e-1)}$ and $\phi^{(e)}_{(X/Y,E/p^e-1)}$ respectively 
by $\L^{(e)}_{(X/Y,\Delta)}$ and $\phi^{(e)}_{(X/Y,\Delta)}$. \par
Next we introduce singularities of pairs defined by the Grothendieck trace map of the Frobenius morphism. 
%
\begin{defn} \label{defn:F-pure} 
Assume that $k$ is perfect. 
Let $X$ be a $k$-scheme of pure dimension satisfying $S_2$ and $G_1$, 
and let $\Delta$ be an effective $\Q$-AC divisor on $X$. Set $Y:=\Spec k$. \\
(1) The pair $(X,\Delta)$ is said to be {\it $F$-pure} 
if for every $e>0$ and for every effective AC divisor $E$ with $E\le (p^e-1)\Delta$, 
the morphism $$\phi^{(e)}_{(X/Y,E/(p^e-1))}:{F_{X/Y}^{(e)}}_*\O_X((1-p^e)K_{X}-E)\to\O_{X_{Y^e}}$$ 
is surjective. We simply say that $X$ is $F$-pure if $(X,0)$ is $F$-pure. \\ 
(2) \textup{\cite[Definition 3.1]{Sch08}} 
Assume that $X$ is a normal variety. 
The pair $(X,\Delta)$ is said to be {\it strongly $F$-regular} 
if for every effective Cartier divisor $D$, there exists an $e>0$ such that 
$$
\phi^{(e)}_{(X/Y,\lceil(p^e-1)\Delta\rceil+D/p^e-1)}:
{F_{X/Y}^{(e)}}_*\O_X(\lfloor(1-p^e)(K_X+\Delta)\rfloor-D)\to\O_{X_{Y^e}}
$$ is surjective. 
Here $\lceil\Delta\rceil$ (resp. $\lfloor\Delta\rfloor$) denotes the round up (resp. down) of $\Delta$. 
We simply say that $X$ is strongly $F$-regular if $(X,0)$ is strongly $F$-regular.  
\end{defn}
\begin{rem}
(1) Usually, the $F$-purity of $X$ is defined by using the absolute Frobenius morphism $F_X$ of $X$. 
In the above definition, by the assumption of the perfectness of $k$, 
we can use the relative Frobenius morphism of the structure morphism $X\to \Spec k$ instead of $F_X$.\\
(2) With the notation as in Definition \ref{defn:F-pure}, we assume that $X$ is normal and affine. 
Then Definition \ref{defn:F-pure} (1) is equivalent to \cite[Definition 2.1]{HW02} (2). 
Indeed, since $\lfloor(p^e-1)\Delta\rfloor\le(p^e-1)\Delta$, 
the condition of Definition \ref{defn:F-pure} implies that 
$\phi^{(e)}_{(X/Y,\lfloor(p^e-1)\Delta\rfloor/(p^e-1))}$ is surjective. 
Conversely, since $\phi^{(e)}_{(X/Y,\lfloor(p^e-1)\Delta\rfloor/(p^e-1))}$ 
factors through $\phi^{(e)}_{(X/Y,E/(p^e-1))}$ 
for every effective Weil divisor $E$ with $E\le\lfloor(p^e-1)\Delta\rfloor$ 
(or equivalently $E\le(p^e-1)\Delta$), 
the surjectivity of $\phi^{(e)}_{(X/Y,\lfloor(p^e-1)\Delta\rfloor/(p^e-1)}$ implies 
the condition of Definition \ref{defn:F-pure}. \\
(3) Let $(X,\Delta)$ be a strongly $F$-regular pair, 
and let $\Delta'$ be an effective $\Q$-divisor on $X$. 
Then there exists an $0<\e\in\Q$ such that $(X,\Delta+\e\Delta')$ is again strongly $F$-regular.
\end{rem}
\begin{lem}\label{lem:relFrob}\samepage
Let $f:X\to Y$ be a projective morphism 
from a pure dimensional $k$-scheme $X$ satisfying $S_2$ and $G_1$ 
to a variety $Y$.
Let $V\subseteq Y$ be a regular open subset such that $f_V:X_V\to V$ is flat.
Let $E$ be an effective AC divisor on $X$ 
whose support does not contain any component of any fiber over $V$, 
and let $a>0$ be an integer not divisible by $p$. 
Set $\Delta:=E\otimes a^{-1}$. 
Assume that $aK_{X_V}+E_V$ is a Cartier divisor on $X_V$
and that there exists a Gorenstein open subset $U\subseteq X$
such that $\codim_{X_{\ol y}}(X_{\ol y}\setminus U_{\ol y})\ge2$ for every $y\in V$.
Then the following holds. \\
$(1)$ \textup{\cite[Corollary~3.31]{PSZ13}} 
The set $V_0:=\{y\in V | (X_{\ol y},\ol{\Delta|_{U_{\ol y}}}) \textup{ is $F$-pure}\}$ is an open subset of $V$. 
Here $\ol y:=\Spec\ol{k(y)}$ and 
$\ol{\Delta|_{U_{\ol y}}}$ is the $\Z_{(p)}$-AC divisor on $X_{\ol y}$ 
obtained as the unique extension of the $\Z_{(p)}$-Cartier divisor $\Delta|_{U_{\ol y}}$ on $U_{\ol y}$.\\
$(2)$ Assume that $V_0$ is non-empty. Let $A$ be a Cartier divisor on $X$ such that $A_{V_0}$ is $f_{V_0}$-ample. 
Then there exists an $m_0>0$ such that 
\begin{align*} 
&{f_{Y^e}}_*(\phi^{(e)}_{(X/Y,E/a)}\otimes\O_{Y^e}(mA_{Y^e}+N_{Y^e})):\\
&{f^{(e)}}_*\O_{X^e}((1-p^e)a^{-1}(aK_{X^e/Y^e}+E)+p^e(mA+N))\to {f_{Y^e}}_*\O_{X_{Y^e}}(mA_{Y^e}+N_{Y^e})
\end{align*}
is surjective over $V_0$ for each $m\ge m_0$, for every Cartier divisor $N$ on $X$ 
whose restriction $N_{V_0}$ to $X_{V_0}$ is $f_{V_0}$-nef and for every $e>0$ with $a|(p^e-1)$. 
\end{lem}
\begin{proof}
Replacing $f:X\to Y$ by $f_V:X_V\to V$, we may assume that $Y$ is regular and $f$ is flat.
Take an integer $e>0$ with $a|(p^e-1)$ and a point $y\in V$.
Then by \cite[Lemma 2.18]{PSZ13} we get 
\begin{align}
\phi^{(e)}_{(X/Y,E/a)}|_{U_{\ol y}}
=\phi^{(e)}_{(U/Y,E|_U/a)}|_{U_{\ol y}}
\cong \phi^{(e)}_{\left(U_{\ol y}/\ol y,E|_{U_y}/a\right)}
:{F_{U_{\ol y}/\ol y}^{(e)}}_*\L^{(e)}_{(U_{\ol y}/\ol y,E|_{U_{\ol y}}/a)}
\to \O_{U_{\ol y^e}}.  \label{align:U} \tag{\ref{lem:relFrob}.1}
\end{align}
Let $\iota_{\ol y}:U_{\ol y}\to X_{\ol y}$ be the open immersion,
and $\ol{E|_{U_{\ol y}}}$ be the unique extension of $E|_{U_{\ol y}}$ to $X_{\ol y}$.
Since $\L^{(e)}_{(X/Y,E/a)}$ is invertible by the assumption, the natural morphism 
$$
\L^{(e)}_{(X/Y,E/a)}|_{X_{\ol y}}
\to{\iota_{\ol y}}_* \L^{(e)}_{(U_{\ol y}/\ol y,E|_{U_{\ol y}}/a)}
\left(=:\L^{(e)}_{(X_{\ol y}/\ol y,\ol{E|_{U_{\ol y}}}/a)}\right)
$$ 
is an isomorphism.
Hence, extending the morphism (\ref{align:U}) to $X_{\ol y}$, we obtain that 
\begin{align}
\phi^{(e)}_{(X/Y,E/a)}|_{X_{\ol y}}
\cong \phi^{(e)}_{\left(X_{\ol y}/\ol y,\ol{E|_{U_y}}/a\right)}
:{F_{X_{\ol y}/\ol y}^{(e)}}_*\L^{(e)}_{(X_{\ol y}/\ol y,\ol{E|_{U_{\ol y}}}/a)}
\to \O_{X_{\ol y^e}}. \label{align:X}\tag{\ref{lem:relFrob}.2}
\end{align}
Then one can show (1) by an argument similar to the proof of \cite[Corollary 3.31]{PSZ13}.
We prove (2). Replacing $Y$ by $V_0$, we may assume that $V_0=V=Y$. 
Then by (\ref{align:X}), we have that 
$\phi^{(e)}_{(X/Y,E/a)}|_{X_{\ol y}}$ is surjective for each $e>0$ with $a|(p^e-1)$ and every $y\in Y$,
which implies that $\phi^{(e)}_{(X/Y,E/a)}$ is surjective for each $e>0$ with $a|(p^e-1)$.
Let $d>0$ be the minimum integer such that $a|(p^d-1)$. 
Note that we have $d|e$ for every integer $e>0$ with $a|(p^e-1)$.
Applying Keeler's relative Fujita vanishing \cite[Theorem 1.5]{Kee03} 
to the kernel of $\phi^{(d)}_{(X/Y,E/a)}$, we obtain an integer $m_0\gg0$ such that 
\begin{align}
{f_{Y^d}}_*\left(\phi^{(d)}_{(X/Y,E/a)}\otimes\left(\O_{X}(mA+N)\right)_{Y^d}\right)\label{align:d}\tag{\ref{lem:relFrob}.3}
\end{align}
is surjective for each $m\ge m_0$ and every $f$-nef Cartier divisor $N$ on $X$.
If necessary, by replacing $m_0$ by some larger integer 
we may assume that $am_0A-(aK_{X/Y}+E)$ is $f$-nef. 
We fix an integer $m\ge m_0$ and an $f$-nef divisor $N$ on $X$.
We show that $$\psi^{(de)}:={f_{Y^{de}}}_*\left(\phi^{(de)}_{(X/Y,E/a)}\otimes\left(\O_{X}(mA+N)\right)_{Y^{de}}\right)$$ 
is surjective for every integer $e>0$ by the induction on $e$. 
We have already seen that $\psi^{(d)}$ is surjective.
We assume that $\psi^{(de)}$ is surjective.
By the definition of $\phi^{(d(e+1))}_{(X/Y,E/a)}$, we have
\begin{align*} 
\psi^{(d(e+1))}
=&{f_{Y^{d(e+1)}}}_*\left(\phi^{(d(e+1))}_{(X/Y,E/a)}\otimes\left(\O_{X}(mA+N)\right)_{Y^{d(e+1)}}\right)\\
\cong& {F_Y^{d}}^*\left({f_{Y^{de}}}_*\left(\phi^{(de)}_{(X/Y,E/a)}\otimes\left(\O_{X}(mA+N)\right)_{Y^{de}}\right)\right)\\ 
&\hspace{50pt} \circ {f^{(de)}_{Y^{d(e+1)}}}_* 
\left(\phi^{(d)}_{(X^{de}/Y^{de},E^{de}/a)}\otimes\left(\L^{(de)}_{(X/Y,E/a)}(p^{de}(mA+N))\right)_{Y^{d(e+1)}}\right) \\
\cong& {F_Y^{d}}^*(\psi^{(de)}) 
\circ {f^{(de)}_{Y^{d(e+1)}}}_* 
\left(\phi^{(d)}_{(X^{de}/Y^{de},E^{de}/a)}\otimes\left(\L^{(de)}_{(X/Y,E/a)}(p^{de}(mA+N))\right)_{Y^{d(e+1)}}\right).
\end{align*}
Since $\psi^{(de)}$ is surjective, ${F_Y^d}^*(\psi^{(de)})$ is also surjective.
We need to show that the morphism 
\begin{align}
{f^{(de)}_{Y^{d(e+1)}}}_* 
\left(\phi^{(d)}_{(X^{de}/Y^{de},E^{de}/a)}\otimes
\left(\L^{(de)}_{(X/Y,E/a)}(p^{de}(mA+N))\right)_{Y^{d(e+1)}}\right)
\tag{\ref{lem:relFrob}.4}\label{align:de}
\end{align}
is surjective.
Here we recall 
that $f^{(de)}:X^{de}\to Y^{de}$ is nothing but $f:X\to Y$.
Let $N_{de,m}$ denote the Cartier divisor 
\begin{align*}
&(p^{de}-1)a^{-1}(amA-(aK_{X^{d}/Y^{d}}+E))+p^{de}N \\
=&(p^{de}-1)(m-m_0)A+(p^{de}-1)a^{-1}(am_0A-(aK_{X^{d}/Y^{d}}+E))+p^{de}N 
\end{align*}
on $X^{d}$. Since $am_0A-(aK_{X/Y}+E)$ is $f$-nef, $N_{de,m}$ is $f^{(d)}$-nef.
Now we can rewrite (\ref{align:de}) as 
\begin{align*}
&{f^{(de)}_{Y^{d(e+1)}}}_*
\left(\phi^{(d)}_{(X^{de}/Y^{de},E^{de}/a)}\otimes\left(
\O_{X^{de}}((1-p^{de})a^{-1}(aK_{X^{de}/Y^{de}}+E^{de})+p^{de}(mA+N))\right)_{Y^{d(e+1)}} \right)\\
=&{f_{Y^d}}_*\left(
\phi^{(d)}_{(X/Y,E/a)}\otimes \left(\O_{X}((1-p^{de})a^{-1}(aK_{X/Y}+E)+p^{de}(mA+N))\right)_{Y^{d}}\right) \\
=&{f_{Y^d}}_*\left(
\phi^{(d)}_{(X/Y,E/a)}\otimes \left(\O_{X}(mA+(p^{de}-1)(mA-a^{-1}(aK_{X/Y}+E))+p^{de}N)\right)_{Y^{d}}\right) \\
= &{f_{Y^d}}_*\left(
\phi^{(d)}_{(X/Y,E/a)}\otimes \left(\O_{X}(mA+N_{de,m})\right)_{Y^{d}}\right).
\end{align*}
Hence the required surjectivity follows from the surjectivity of (\ref{align:d}).
\end{proof}
\section{Weak positivity over $S$ and a numerical invariant}\label{section:inv}
In this section, we define an invariant of coherent sheaves on normal varieties of positive characteristic which measures positivity. This will play an important role in the proof of the main theorem. \par 
We first recall some definitions of the positivity of coherent sheaves on normal varieties over an algebraically closed field $k$ of arbitrary characteristic. 
\begin{defn} \label{defn:positivity} 
Let $Y$ be a quasi-projective normal variety over a field $k$, 
let $\G$ be a coherent sheaf on $Y$, and let $H$ be an ample Cartier divisor. 
Let $V\subseteq Y$ be the largest open subset such that 
$\G|_{V}$ is locally free, and let $S$ be a non-empty subset of $V$. 
\begin{itemize} 
\item[$($i$)$] 
$\G$ is said to be \textit{globally generated over $S$} if the natural morphism $H^0(Y,\G)\otimes_k\O_Y\to\G$ is surjective over $S$. 
\item[$($ii$)$] 
$\G$ is said to be \textit{weakly positive over $S$} if for every integer $a>0$, 
there exists an integer $b>0$ such that $(S^{ab}\G)^{**}\otimes\O_Y(b H)$ is globally generated over $S$. 
Here $S^{ab}(\underline{~~})$ and $(\underline{~~})^{**}$ denote the $ab$-th symmetric product and the double dual, respectively. 
\item[$($iii$)$] 
$\G$ is said to be \textit{big over $S$} if there exists an integer $a>0$ such that $(S^a\G)(-H)$ is weakly positive over $S$. 
\end{itemize}
We simply say that $\G$ is globally generated (resp. weakly positive, big) over $y$ 
when $S=\{y\}$ for a point $y\in V$. 
$\G$ is said to be \textit{generically globally generated} 
if $\G$ is globally generated over the generic point $\eta$ of $Y$.
$\G$ is said to be \textit{weakly positive} (resp. \textit{big}) 
if it is weakly positive (resp. big) over $\eta$.
\end{defn}
\begin{rem} 
The notion of weak positivity is first introduced by Viehweg as a generalization of nefness of vector bundles, 
when $S$ is an open subset \cite{Vie95}. In \cite{Kol87}, \cite{Pat13} and \cite{Eji15}, (resp. \cite{Nak04}), 
this notion is also defined in the case when $S=\{\eta\}$ (resp. $S=\{y\}$ for a point $y\in Y$). 
\end{rem}
\begin{rem}\label{rem:wp} 
Let $Y,\G,S$ and $H$ be as above. \\
%
(1) The above definition is independent of the choice of $H$ 
(cf. \textup{\cite[Lemma 2.14]{Vie95}}).  \\ 
(2) Let $Y_0\subseteq Y$ be an open subset containing $S$ such that $\codim_Y(Y\setminus Y_0)\ge2$ 
and let $i:Y_0\to Y$ be the open immersion. 
Then we have
$$(S^{m}\G)^{**}(nH)\cong {i}_*\left(((S^{m}\G)^{**}(nH))|_{Y_0}\right)
\cong {i}_*\left((S^{m}(\G|_{Y_0}))^{**}(nH|_{Y_0})\right)$$
for each integers $m,n$ with $m>0$.
Therefore, we see that $\G$ is weakly positive (resp. big) over $S$ 
if and only if so is $\G|_{Y_0}$. \\
(3) 
The natural morphism $\G\to\G^{**}$ induces the morphism 
$(S^{m}\G)^{**}(nH)\to(S^{m}\G^{**})^{**}(nH)$ for each integers $m,n$ with $n>0$, 
which is an isomorphism because of the normality of $Y$. 
In particular, $\G$ is weakly positive (resp. big) over $S$ if and only if so is $\G^{**}$. \\
(4) Assume that $\G$ is a vector bundle and that $Y$ is a smooth projective curve. 
Then $\G$ is weakly positive (resp. big) over $Y$ if and only if $\G$ is nef (resp. ample).  \\
(5) Assume that $\G$ is a line bundle and that $Y$ is projective. 
Then $\G$ is weakly positive (resp. big) over the generic point $\eta$ of $Y$ if and only if $\G$ is pseudo-effective (resp. big). 
\end{rem}
\begin{obs}\label{obs:wp}
(1) With the notation as Definition \ref{defn:positivity}, we define
\begin{align*}
T'_S(\G,H)&:=\begin{array}{cc|ccl}
\ldelim\{{3}{0.1pt}&& \textup{there exist $a,b\in\Z$ such that}& \rdelim\}{3}{1pt} \cr
&\e\in\Q&\textup{ $\e=a/b$, $b>0$, and $(S^{b}\G)^{**}(-aH)$ is}& \cr
&                 &\textup{globally generated over $S$.}& 
\end{array}~, \textup{ and} \\
t'_S(\G,H)&:=\sup T'_S(\G,H).
\end{align*}
We first prove that $T'_S(\G,H)$ is equal to $\Q\cap (-\infty,t'_S(\G,H))$ or $\Q\cap (-\infty,t'_S(\G,H)]$.
By Remark \ref{rem:wp} (3), we have $T'_S(\G,H)=T'_S(\G^{**},H)$.
Furthermore, similarly to Remark \ref{rem:wp} (2), we see that $T'_S(\G,H)=T'_S(\G|_V,H|_V)$, 
where $V\subseteq Y$ is the maximum open subset such that $\G|_V$ is locally free.
Hence we may assume that $\G$ is locally free.
If $(S^{b}\G)(-aH)$ is globally generated over $S$ for integers $a,b$ with $b>0$, 
then $(S^{bc}\G)(-acH)$ is also globally generated over $S$ for every $c>0$,
because of the natural morphism $$S^c((S^b\G)(-aH))\to (S^{bc}\G)(-acH)$$ which is surjective over $S$.
Then $(S^{bc}\G)((-ac+d)H)$ is also globally generated over $S$ for every $d>0$ such that $dH$ is free.
Hence we see that $(ac-d)/(bc)=a/b-d/(bc)\in T'_S(\G,H)$, which proves our claim. \\
(2) Next, we show that $\G$ is weakly positive (resp. big) over $S$ 
if and only if $t'_S(\G,H)\ge0$ (resp $>0$).
By an argument similar to the above, we may assume that $\G$ is locally free.
The definition of the weak positivity of $\G$ over $S$ is equivalent to that 
$-1/a\in T'_S(\G,H)$ for all $a>0$, which is also equivalent to $t'_S(\G,H)\ge 0$ because of (1). 
If $\G$ is big, then there exists an integer $c>0$ such that $(S^c\G)(-H)$ is weakly positive. 
Then for an $a>0$ there exists a $b\gg0$ such that $(S^{ab}(S^c\G))(-abH+aH)$ is globally generated over $S$, 
and so is $(S^{abc}\G)(a(1-b)H)$ as the argument in (1). 
Hence $t'_S(\G,H)\ge (ab-a)/(abc)=1/c-1/(bc)>0$. 
Conversely, if $t'_S(\G,H)>0$ then by (1) we have integers $a,b>0$ such that 
$(S^{b}\G)(-aH)$ is globally generated over $S$,
and hence $\G$ is big. \\
(3) For a line bundle $\L$ on $Y$ and an integer $m>0$, it is easily seen that $t'_S(\L^m,H)=mt'_S(\L,H)$. 
Hence we see that $\L$ is weakly positive over $S$ if and only if so is $\L^m$.
\end{obs}
By Observation \ref{obs:wp} (3), we can define the weak positivity of a $\Q$-Cartier divisor.
\begin{defn} \label{defn:wp_div}
With the notation as Definition \ref{defn:positivity}, let $D$ be a $\Q$-Cartier divisor on $Y$ and $m>0$ be an integer such that $mD$ is Cartier.
Then $D$ is said to be weakly positive (resp. big) over $S$ if so is $\O_X(mD)$.
\end{defn}
Note that this definition is independent of the choice of $m$ by Observation \ref{obs:wp} (3). \par
In the rest of this section, we suppose that the characteristic of the base field $k$ is positive.
\begin{defn} \label{defn:inv} 
Let $Y,\G,S$ be as in Definition \ref{defn:positivity}, and assume that the characteristic of $k$ is $p>0$. 
Let $D$ be a $\Q$-Cartier divisor. Then we define  
\begin{align*} 
T_S(\G,D)&:= \begin{array}{cc|ccl}
\ldelim\{{3}{0.1pt}&& \textup{there exists an $e>0$ such that}& \rdelim\}{3}{1pt} \cr
&\e\in\Q&\textup{ $p^e\e D$ is Cartier and $({F_Y^e}^*\G)(-p^e\e D)$ is}& \cr
&                 &\textup{globally generated over $S$.}& 
\end{array}~,\textup{ and}\\
t_S(\G,D)&:=\sup T_S(\G,D)\in \R\cup\{-\infty,+\infty\}.  
\end{align*} 
\end{defn}
\begin{lem} \label{lem:ts} 
Under the same assumption as the above, let $\F$ be a coherent sheaf on $Y$.  
\begin{itemize} 
\item[$(1)$] If there exists a morphism $\F\to\G$ such that it is surjective over $S$, 
then $t_S(\F,D)\le t_S(\G,D)$.  
\item[$(2)$] Assume that $\{t_S(\F,D),t_S(\G,D)\}\ne\{-\infty,+\infty\}$. 
Then $$t_S(\F,D)+t_S(\G,D) \le t_S(\F\otimes\G,D).$$  
\item[$(3)$] For each $e>0$, $t_S({F_Y^e}^*\G,D)=p^et_S(\G,D)$. 
\item[$(4)$] If the rank of $\G$ is positive, and if $t_S(\G,D)=+\infty$, 
then $-D$ is weakly positive over $S$. 
\end{itemize} 
\end{lem}
\begin{proof} 
(1)--(3) follow directly from the definition. 
We prove (4). Recall that for every $y\in S$, the stalk $\G_y$ of $\G$ at $y$ is free of positive rank. 
From this we see that the natural morphism $\G^*\otimes\G\to\O_Y$ is surjective over $S$.
Furthermore, there is an ample Cartier divisor $H$ such that $\G^*(H)$ is globally generated, 
and thus we have a surjective morphism $\O_Y^{\oplus h}\to\G^*(H)$ for some $h>0$. 
Hence we get a morphism $$\G^{\oplus h}\to\G^*\otimes\G(H)\to\O_Y(H)$$ which is surjective over $S$. 
By (1) we have $t_S(\O_Y(H),D)=+\infty$, and thus there exists a sequence $\{\e_n>0\}_{n\ge 1}$ of elements of $T_S(\O_Y(H),D)$ 
such that $\e_n\xrightarrow{n\to+\infty}+\infty$. 
By the definition of $t_S(\O_Y(H),D)$, for every $n\ge1$ there exists $e\ge1$ 
$$({F^e_Y}^*\O_Y(H))(-\e_np^eD)\cong \O_Y(p^e(H-\e_nD)) $$ is globally generated over $S$. 
Set $\G:=\O_Y(-bD)$ for an integer $b>0$ such that $bD$ is Cartier. 
Then for an integer $l_n>0$ such that $\e_nl_np^e$ is an integer, 
$$(S^{\e_nl_np^e}\G)^{**}\otimes\O_Y(bl_np^eH)\cong \O_Y(\e_nl_np^e(-bD)+bl_np^eH)\cong S^{bl_n}(\O_Y(p^e(H-\e_nD)))$$ 
is also globally generated over $S$. Using the notation of Observation \ref{obs:wp}, we see that 
$(-bl_np^e)/(\e_nl_np^e)=-b/\e_n\le t'_S(\O_Y(-bD),H)$, and so $0\le t'_S(\O_Y(-bD),H)$. 
As shown in Observation \ref{obs:wp} (2), this implies that $\O_Y(-bD)$ is weakly positive over $S$. 
\end{proof}
\section{Main theorem}\label{section:thmp}
The purpose of this section is to prove Theorems \ref{thm:thmp} and \ref{thm:thmp2}. 
First we prepare three lemmas. The following lemma is used in the proof of Theorem \ref{thm:thmp}.
\begin{lem}\label{lem:locfree}
Let $k$ be a field.
Let $f:X\to Y$ be a surjective projective morphism from a $k$-scheme $X$ to a variety $Y$, 
let $A$ be an $f$-ample Cartier divisor on $X$, and let $\F$ be a coherent sheaf on $X$. 
Then there exists an integer $m_0>0$ such that for each integer $m\ge m_0$ 
and for every nef Cartier divisor $N$ on $X$,
$f_*(\F(mA+N))$ is locally free over $V$, 
where $V\subseteq Y$ is an open subset such that $\F_V$ is flat over $V$.
\end{lem}
\begin{proof}
By Keeler's relative Fujita vanishing (\cite[Theorem 1.5]{Kee03}), 
there exists an integer $m_0>0$ such that 
$$Rf_*^i\F(mA+N)=0$$
for each $m\ge m_0$, for every nef Cartier divisor $N$ on $X$ and for each $i>0$. 
Fix an integer $m\ge m_0$ and a nef Cartier divisor $N$.
For each $i\ge0$, we define the function $h^i$ on $V$ 
by $h^i(y):=\dim_{k(y)} H^0(X_y,\F(mA+N)|_{X_y})$.
Since $\dim X_y\le\dim X$ for every $y\in V$, 
we see that $h^i=0$ for each $i>\dim X$.
Let $i\ge2$ be an integer such that $h^i=0$. 
By applying \cite[Theorem I\hspace{-1pt}I\hspace{-1pt}I 12.11]{Har77} to 
the morphism $f_{V}:X_{V}\to V$ and $\F(mA+N)_{V}$, we obtain that $h^{i-1}=0$. 
From this, we see that $h^i=0$ for each $i\ge1$, 
and hence $\chi(\F(mA+N)|_{X_y})=h^0(y)$ for every $y\in V$.
Since the left hand side is constant on $V$ by 
\cite[Theorem I\hspace{-1pt}I\hspace{-1pt}I 9.9]{Har77} and its proof, 
$h^0$ is also constant on $V$. 
Applying \cite[Corollary I\hspace{-1pt}I\hspace{-1pt}I 12.9]{Har77}, 
we obtain that $f_*\F(mA+N)$ is locally free over $V$, which is our claim.
\end{proof}
\begin{lem}\label{lem:glgen} 
With the notation as in Lemma \ref{lem:locfree}, 
assume that $X$ and $Y$ are projective and $A$ is ample. 
Then there exists an integer $m_0>0$ such that for each integer $m\ge m_0$ 
and for every nef Cartier divisor $N$ on $X$, 
$f_*(\F(mA+N))$ is globally generated.
\end{lem}
\begin{proof} 
Let $H$ be an ample and free Cartier divisor on $Y$, 
and let $m_1>0$ be an integer such that $m_1A-\dim Yf^*H$ is nef. 
By the Fujita vanishing (\cite[Theorem (1)]{Fuj82}, \cite[Section 5]{Fuj83}) and 
Keeler's relative Fujita vanishing (\cite[Theorem 1.5]{Kee03}), there exists an integer $m_2>0$ such that 
$$H^i(X,\F(mA+N))=0 ~\textup{ and } ~R^if_*(\F(mA+N))=0$$ for each $m\ge m_2$, 
for every nef Cartier divisor $N$ on $X$ and for each $i>0$. 
Since $R^if_*(\F(mA+N))=0$, by the Leray spectral sequence, we have 
\begin{align*}
H^i(Y,(f_*(\F((m+m_1)A+N)))(-iH))
&\cong H^i(Y,f_*(\F((m+m_1)A-if^*H+N))) \\
&\cong H^i(X,\F(mA+m_1A-if^*H+N))
\end{align*}
for each $m\ge m_0$ and for each $i>0$.  
Since $m_1A-if^*H+N$ is nef for $0<i\le\dim Y$, the right term vanishes.
This implies that $f_*(\F((m+m_1)A+N))$ is $0$-regular with respect to $H$, 
and hence it is globally generated by the Castelnuovo-Mumford regularity \cite[Theorem 1.8.5]{Laz04}. 
Defining $m_0:=m_1+m_2$, we obtain the assertion.
\end{proof}
The next lemma is a consequence of the relative Castelnuovo-Mumford regularity \cite[Example 1.8.24]{Laz04},
and is used in the proof of Theorems \ref{thm:thmp} and \ref{thm:thmp2}.
\def\N{{\mathcal N}}
\begin{lem} \label{lem:mult}
Let $k$, $f$, $X$, $Y$ and $\F$ be as in Lemma \ref{lem:locfree}. 
Let $W\subseteq Y$ be an open subset. Let $L$ be a Cartier divisor on $X$.
\begin{itemize}
\item[$(1)$] 
If $L_W$ is $f_W$-free, then there exists an integer $n_0>0$ such that 
for each $n\ge n_0$ and each $m>0$, the natural morphism 
$$f_*\O_X(mL)\otimes f_*(\F(nL))\to f_*(\F((m+n)L))$$ 
is surjective over $W$.
\item[$(2)$] 
If $L_W$ is $f_W$-ample and $f_W$-free, then there exists an integer $n_0>0$ 
such that for each $n\ge n_0$, for each $m>0$ and 
for every Cartier divisor $N$ on $X$ whose restriction $N_W$ to $X_W$ is $f_W$-nef, 
the natural morphism 
$$f_*\O_X(mL)\otimes f_*(\F(nL+N))\to f_*(\F((n+m)L+N))$$ 
is surjective over $W$.
\end{itemize}
\end{lem}
\begin{proof}
\def\M{{\mathcal M}}
Replacing $f:X\to Y$ by $f_W:X_W\to W$, we may assume that $W=Y$.
Set $\L:=\O_X(L)$. We first show that (2) implies (1). Indeed, since $\L$ is $f$-free, 
there exist surjective projective morphisms $g:X\to Z$ and $h:Z\to Y$ with $h\circ g=f$
such that $\L\cong g^*\M$ for an $h$-ample and $h$-free line bundle $\M$ on $Z$. 
Then we have the following commutative diagram of natural morphisms:
\begin{align*}
\xymatrix@R=25pt@C=15pt{ (h_*\M^m)\otimes h_*((g_*\F)\otimes \M^n) \ar[r] \ar[d] & 
(h_*g_*g^*\M^m) \otimes h_*((g_*\F)\otimes \M^n) \ar[r]^(0.55){\cong} &
h_*g_*\L^m \otimes h_*g_*(\F\otimes \L^n) \ar[d] \\
h_*((g_*\F)\otimes\M^{m+n}) \ar[rr]^{\cong} & & h_*g_*(\F\otimes \L^{m+n}) }
\end{align*}
Here the isomorphisms follow from the projection formula. 
The surjectivity of the left vertical morphism induces that of the right vertical morphism.
Hence we see that it is enough to prove (2). \par
We show (2). We first prove the case $m=1$. 
In this case, by Keeler's relative Fujita vanishing (\cite[Theorem 1.5]{Kee03}), 
there exists an integer $n_0>0$ such that for each $n\ge n_0$ and every $f$-nef line bundle $\N$ on $X$, 
$\F\otimes\L^n\otimes\N$ is $0$-regular with respect to $\L$ and $f$,
and hence the surjectivity follows from the relative Castelnuovo-Mumford regularity \cite[Example 1.8.24]{Laz04}.  
Next we show the case when $m\ge2$.
In this case, we have the following commutative diagram of natural morphisms:
\begin{align*}
\xymatrix@R=25pt@C=10pt{ (f_*\L)^{\otimes m}\otimes f_*(\F\otimes\L^n\otimes \N) \ar[r] \ar[d] & 
(f_*\L)^{\otimes m-1} \otimes f_*(\F\otimes\L^{n+1}\otimes\N) \ar[r] & \cdots \ar[r] & f_*(\F\otimes\L^{m+n}\otimes\N)  \ar@{=}[d] \\
f_*\L^m\otimes f_*(\F\otimes\L^n\otimes\N) \ar[rrr] & & & f_*(\F\otimes\L^{m+n}\otimes\N)}
\end{align*}
By the above argument, we see that the upper horizontal morphisms are surjective, 
and hence so is the lower horizontal morphism, which is our claim.
\end{proof}
\begin{notation}\label{notation:main}
Let $k$ be an $F$-finite field and 
$f:X\to Y$ be a surjective projective morphism 
from a pure dimensional quasi-projective $k$-scheme $X$ satisfying $S_2$ and $G_1$ 
to a normal quasi-projective variety $Y$. 
Let $a>0$ be an integer and $E$ be an effective AC-divisor on $X$ such that $aK_X+E$ is Cartier.
Set $\Delta:=E\otimes a^{-1}$. 
Let $U\subset X$ be a Gorenstein open subset and 
let $S\subseteq Y$ be a non-empty subset such that the following conditions hold for every $y\in S$:
\begin{itemize} 
\item[$($i$)$] $Y$ is regular in a neighborhood of $y$ and $f$ is flat at every point in $f^{-1}(y)$.
\item[$($ii$)$] $\codim_{X_{\ol y}}(X_{\ol y}\setminus U_{\ol y})\ge2$.
\item[$($iii$)$] The support of $E$ does not contain any irreducible component of $f^{-1}(y)$.
\item[$($iv$)$] $(X_{\ol y},\ol{\Delta|_{U_{\ol y}}})$ is $F$-pure, 
where $\ol y:=\Spec \ol{k(y)}$ and
$\ol{\Delta|_{U_{\ol y}}}$ is the $\Z_{(p)}$-AC divisor on $X_{\ol y}$ 
obtained as the unique extension of the $\Z_{(p)}$-Cartier divisor $\Delta|_{U_{\ol y}}$ on $U_{\ol y}$.
\end{itemize}
Let $D$ be a $\Q$-Cartier divisor on $Y$. 
\end{notation}
In this setting, we show the following theorems.
\begin{thm}\label{thm:thmp}\samepage
With the notation as in \ref{notation:main}, 
assume that $X$ is projective and $K_X+\Delta$ is $\Q$-Cartier.
\begin{itemize}
\item[$($1$)$] If $p\nmid a$ and $-(K_X+\Delta+f^*D)$ is nef, 
then $-(K_Y+D)$ is weakly positive over an open subset of $Y$ containing $S$.
\item[$($2$)$] If $K_X$ is $\Q$-Cartier and $-(K_X+\Delta+f^*D)$ is ample, 
then $-(K_Y+D)$ is big over an open subset of $Y$ containing $S$.
\end{itemize}
\end{thm}
\begin{thm}\label{thm:thmp2}\samepage
With the notation as in \ref{notation:main}, 
let $b>0$ be an integer such that $bD$ is Cartier. Set $L:=abK_X+bE+af^*(bD)$.
Assume that the natural morphism 
$$f^*f_*\O_X(-L)\to\O_X(-L)$$ 
is surjective over $f^{-1}(S)$.
If $p\nmid a$ and $f_*\O_X(-L)$ is globally generated over $S$, 
then $-(K_Y+D)$ is weakly positive over $S$.
\end{thm}
\begin{rem}
In the case when $X$ is a normal variety and $S$ is 
the singleton $\{\eta\}$ of the generic point $\eta$ of $Y$, 
the conditions (i)--(iii) above hold. However, 
the condition (iv) does not necessarily hold even if $X$ is smooth and $\Delta=0$. 
\end{rem}
\begin{proof}[Proof of Theorem $\ref{thm:thmp}$]
First we show that (1) implies (2). 
By the assumption of (2), $\Delta$ is $\Q$-Cartier. 
Let $m>0$ and $c\ge0$ be integers such that $a=mp^c$ and $p\nmid m$. 
Take an integer $e\gg0$. Set $a':=m(p^e+1)$, $E':=p^{e-c}E$ and $\Delta':=(p^{e-c}E)\otimes {a'}^{-1}$. 
Then $p\nmid a'$, and $E'$ satisfies the assumption (iii). 
Furthermore, $\Delta'$ satisfies the assumption (iv). Indeed, we have 
\begin{align*}
\Delta-\Delta'&=a'E\otimes\frac{1}{aa'}-ap^{e-c}E\otimes\frac{1}{aa'} \\
&=(m(p^{e}+1)-mp^cp^{e-c})E\otimes\frac{1}{aa'} 
=mE\otimes\frac{1}{aa'}
=\frac{m}{a'}\Delta=\frac{1}{p^e+1}\Delta \ge0, 
\end{align*}
and hence $\Delta'_{\ol y}\le \Delta_{\ol y}$ for every $y\in S$, 
which implies that $(X_{\ol y},\Delta'_{\ol y})$ is also $F$-pure for every $y\in S$. 
Replacing $e$ by a larger integer if necessary, we may assume that 
$$-(K_X+\Delta'+f^*D)=-(K_X+\Delta+f^*D)-\frac{1}{p^e+1}\Delta$$ is an ample $\Q$-Cartier divisor. 
Then $-(K_X+\Delta'+f^*(D'+\e H))$ is nef for an ample Cartier divisor $H$ on $Y$ 
and an $\e\in\Q$ with $0<\e\ll1$. 
By (1) we see that $-(K_Y+D+\e H)$ is weakly positive over an open subset $Y_1\subseteq Y$ containing $S$, 
and hence $-(K_Y+D)$ is big over $Y_1$. \par
Next we show (1). The proof is divided into seven steps. \\
\setcounter{step}{0}
\begin{step}
In this step, we prove that we may assume that $S$ is an open subset. 
In other words, there exists an open subset $S'\subseteq Y$ containing $S$ such that 
the conditions (i)--(iv) hold for every $y\in S'$. 
Note that weak positivity over $S'$ obviously implies weak positivity over $S$.
Set $r:=\dim X-\dim Y$. Let $Y'$ be the subset of points satisfying condition (i). 
Then $Y'\subseteq Y$ is an open subset containing $S$, 
and $X_{\ol y}$ is an $S_2$ scheme of pure dimension $r$ for every $y\in Y'$.
Since the function $\dim (X\setminus U)_{\ol y}$ on $Y$ is upper semicontinuous 
by Chevalley's theorem (\cite[Corollaire 13.1.5]{Gro65}), 
we obtain that the function 
$
\codim_{X_{\ol y}}(X_{\ol y}\setminus U_{\ol y})
=r-\dim \left((X\setminus U)_{\ol y}\right)
$ 
on $Y'$ is lower semicontinuous. 
Therefore the subset of points $y$ in $Y'$ with 
$\codim_{X_{\ol y}}(X_{\ol y}\setminus U_{\ol y})\ge 2$ 
is open, and contains $S$.
By an argument similar to the above, 
we see that the subset of points $y$ in $Y'$ with $\codim_{X_{\ol y}}(E_{\ol y})\ge1$ is open, and contains $S$.
Hence, shrinking $Y'$ if necessarily, we may assume that every $y\in Y'$ satisfies the conditions (i)--(iii).
Applying Lemma \ref{lem:relFrob} (1) (set $S'=V$), we obtain an open subset 
$S'\subseteq Y'$ containing $S$ such that the conditions (i)--(iv) hold for every $y\in S'$. 
\end{step}
\begin{step}
In this step we reduce the proof to a numerical condition.
By Step 1, we may assume that $S\subseteq Y$ is open.
Let $A$ be an ample and free Cartier divisor on $X$. 
By Lemma \ref{lem:glgen}, we have an integer $m'>0$ such that 
$f_*\O_X(mA)$ is globally generated for each integer $m\ge m'$.
Let $V$ denote the regular locus of $Y$. Note that $S\subseteq V$ by the assumption (i). 
By Lemma \ref{lem:locfree}, we have an integer $m''>0$ such that 
${f_V}_*\O_{X_V}(N+mA_V)$ is locally free over $S$ for each integer $m\ge m''$
and every Cartier divisor $N$ on $X_V$ whose restriction $N_S$ to $X_S$ is $f_S$-nef.
Replacing $A$ by $\max\{m',m''\}A$ if necessary, we may assume that $m'=m''=1$.
For simplicity, we set 
\begin{align*}
D'&:=D+K_Y, \\
\G(l,m)&:={f_V}_*\O_{X_V}(la^{-1}(aK_{X_V/V}+E_{V})+mA_{V})\quad \textup{and}\\
t(l,m)&:=t_S(\G(l,m),D'|_{V})
\end{align*} 
for every integers $l,m$ with $a|l$. 
Then we have $t(0,m)\ge0$ for each $m>0$ by the global generation of $f_*\O_X(mA)$.
Furthermore, since $-la^{-1}(aK_{X_V/V}+E_V)_S$ is $f_S$-nef 
for each $l\ge0$ with $a|l$ by the assumption,
$\G(-l,m)$ is locally free over $S$ for each $l\ge0$ with $a|l$ and $m>0$.
Our goal is to prove that $t(0,m)=+\infty$ for an integer $m\ge0$. 
If this is shown, then by Lemma \ref{lem:ts} we obtain that 
$-D'|_{V}=-(K_Y+D)|_{V}$ is weakly positive over $S$. 
As mentioned in Remark \ref{rem:wp} (2), this is equivalent to the weak positivity of $-(K_Y+D)$. 
In order to get a contradiction, we assume that $t(0,m)\ne+\infty$ for each integer $m>0$. 
Then we have $0\le t(0,m)\in\R$.  
\end{step}
\begin{step}
Let $d>0$ be an integer divisible by $a$ such that $dD$ and $da^{-1}(aK_X+E)$ is Cartier. 
Let $q_l$ and $r_l$ denote the quotient and the remainder 
of the division of an integer $l$ by $d$, respectively. 
In this step, we show that there exists an integer $m_0>0$ such that $$dq_l\le t(-l,m)$$ 
for each $m\ge m_0$ and each $l\ge0$ with $a|l$. 
Since $aK_X+E$ is Cartier and $S$ is contained in the regular locus $V$ of $Y$, 
we have that for each $0\le r<d$ with $a|r$, \def\M{{\mathcal{M}}}
$$\M_r:=f^*\o_Y^{\otimes r}(-ra^{-1}(aK_X+E))$$ is invertible in a neighborhood of $f^{-1}(S)$.
Hence we see that $\M_r$ is flat over $Y$ at every point of $f^{-1}(S)$ by the assumption (i).
As shown in Lemmas \ref{lem:locfree} and \ref{lem:glgen}, 
there exists an integer $m_0>0$ such that for every integer $m\ge m_0$, 
for every nef Cartier divisor $N$ on $X$ and for each $0\le r < d$ with $a|r$, 
$f_*(\M_r(N+mA))$ is locally free over $S$ and globally generated over $Y$. 
Set $L:=-(dK_X+da^{-1}E+f^*(dD))$. Then by the assumption of (1), $L$ is a nef Cartier divisor on $X$.
For each $m\ge m_0$ and each $l\ge0$ with $a|l$, we have the following isomorphisms:
\begin{align*}
&\left(\G(-l,m)\right)(-dq_lD'|_{V})\\
= &\left({f_V}_*\O_{X_V}(-la^{-1}(aK_{X_V/V}+E_{V})+mA_{V})\right)(-dq_lD'|_{V}) \\
\cong &\left({f_V}_*\O_{X_V}((-r_l-dq_l)a^{-1}(aK_{X_V/V}+E_{V})+mA_{V})\right)(-q_l(dD'|_{V})) \\
\cong &{f_V}_*\O_{X_V}(-r_la^{-1}(aK_{X_V/V}+E_{V})-q_l(dK_{X_V/V}+da^{-1}E_{V}+{f_V}^*(dD'|_{V}))+mA_{V}) \\
\cong &{f_V}_*\O_{X_V}(-r_la^{-1}(aK_{X_V/V}+E_{V})-q_l(dK_{X_V}+da^{-1}E_{V}+{f_V}^*(dD|_{V}))+mA_{V}) \\
\cong &{f_V}_*\O_{X_V}(r_l{f_V}^*K_{V}-r_la^{-1}(aK_{X}+E)_{V}-q_l(dK_{X}+da^{-1}E+{f}^*(dD))_{V}+mA_{V}) \\
\cong &{f_V}_*\left((\M_{r_l}|_{X_V})(q_lL_{V}+mA_{V})\right) 
\cong \left(f_*\M_{r_l}(q_lL+mA)\right)|_{V} 
\end{align*} 
Here the isomorphisms in the forth and the last line follow from 
the projection formula and the flatness of $V\to Y$, respectively. 
Note that since $a|l$ and $a|d$, we have $a|r_l$. 
By the choice of $m_0$, we see that $f_*(\M_{r_l}(q_lL+mA))$ is globally generated, 
and hence so is $\left(\G(-l,m)\right)(-dq_lD'|_{V})$.
This implies that $dq_l\le t(-l,m)$. 
\end{step}
\begin{step}
The assumption of the projectivity of $X$ is used only in Steps 2 and 3. 
Hence, replacing $f:X\to Y$ be $f_V:X_V\to V$, we may assume that $V=Y$ in the steps below. 
Then
\begin{align*} 
\G(l,m)=f_*\O_X(la^{-1}(aK_{X/Y}+E)+mA)\quad \textup{and} \quad 
t(l,m)=t_S(\G(l,m),D') 
\end{align*}
for every integers $l,m$ with $a|l$. 
\end{step}
\begin{step}
In this step, we show that there exists an integer $n_0> 0$ such that 
\begin{align*} 
t(0,m)+t(-l,n)\le t(-l,m+n) 
\end{align*} 
for each $n\ge n_0$, each $m>0$ and each $l\ge0$ with $a|l$.  
By Step 1 we may assume that $S\subseteq Y$ is open. 
By the choice of $A$, we have that $A_S$ is $f_S$-ample and $f_S$-free.  
Applying Lemma \ref{lem:mult} (2),
there exists an integer $n_0>0$ such that for each $n\ge n_0$, each $m>0$ 
and every Cartier divisor $N$ on $X$ whose restriction $N_S$ to $X_S$ is $f_S$-nef, 
\begin{align*}
{f_S}_*\O_{X_S}(mA_S) \otimes {f_S}_*\O_{X_S}(N_S+nA_S)
\to {f_S}_*\O_{X_S}(N_S+(m+n)A_S)
\end{align*} 
is surjective. Since $-(aK_{X/Y}+E)_{S}$ is an $f_{S}$-nef Cartier divisor, 
it follows that for each $n\ge n_0$ and each $l\ge0$ with $a|l$, 
the natural morphism $$\G(0,m)\otimes\G(-l,n) \to \G(-l,m+n)$$ is surjective over $S$.
Hence 
we obtain that 
\begin{align*} 
t(0,m)+t(-l,n) 
\overset{\textup{\hspace{54pt}}}{=} &t_S(\G(0,m),D')+t_S(\G(-l,n),D') \\
\overset{\textup{Lemma \ref{lem:ts} (2)}}{\le} &t_S(\G(0,m)\otimes\G(-l,n),D') \\
\overset{\textup{Lemma \ref{lem:ts} (1)}}{\le} &t_S(\G(-l,m+n),D')=t(-l,m+n).
\end{align*} 
Here note that since $t(0,m)\ne+\infty,-\infty$ as mentioned in Step 2, we can use Lemma \ref{lem:ts} (2).
\end{step}
\begin{step}
In this step, we prove that there exists an integer $m_1>0$ such that 
$$t(1-p^e,mp^e)\le p^et(0,m)$$ for each $m\ge m_1$ and each $e\ge0$ with $a|(p^e-1)$. 
Now we have the following morphism: 
$${f_{Y^e}}_*\phi^{(e)}_{(X/Y,E/a)}\otimes\O_{X_{Y^e}}(mA_{Y^e}): 
\G(1-p^e,mp^e)\to{f_{Y^e}}_*\O_{X_{Y^e}}(mA_{Y^e})\cong{F_Y^e}^*\G(0,m)$$ 
for each integers $e,m\ge0$ with $a|(p^e-1)$. 
Here the isomorphism follows from the flatness of $F_Y^e$.  
Applying Lemma \ref{lem:relFrob} (2),
we see that there exists an integer $m_1>0$ such that the morphism 
\begin{align*} 
{(f_S)_{S^e}}_*\left(\phi^{(e)}_{(X_S/S,E_S/a)}\otimes\O_{X_{S^e}}(mA_{S^e})\right)
\cong \left({f_{Y^e}}_*\phi^{(e)}_{(X/Y,E/a)}\otimes\O_{X_{Y^e}}(mA_{Y^e})\right)|_S 
\end{align*} 
is surjective for each $m\ge m_1$ and each $e\ge0$ with $a|(p^e-1)$. 
Therefore 
we obtain that 
\begin{align*} 
t(1-p^e,mp^e) =t_S(\G(1-p^e,mp^e),D') 
\overset{\textup{Lemma \ref{lem:ts} (1)}}{\le} &t_S({F_Y^e}^*\G(0,m),D') \\
\overset{\textup{Lemma \ref{lem:ts} (3)}}{=} &p^et_S(\G(0,m),D')=p^et(0,m).
\end{align*}
\end{step}
\begin{step}[7] 
In this step, we obtain a contradiction.
Fix an integer $\mu>0$ such that $m_1\mu\ge \max\{m_0,n_0\}$. Then by Steps 5 and 6, we see that 
\begin{align*} 
(p^e-\mu)t(0,m_1) + t(1-p^e,m_1\mu)
\overset{\textup{Step 5}}{\le} t(1-p^e,m_1p^e) 
\overset{\textup{Step 6}}{\le} p^et(0,m_1) 
\end{align*} 
for each $e\ge0$ with $a|(p^e-1)$ and $p^e\ge\mu$. 
Therefore, we obtain that $t(1-p^e,m_1\mu)\le \mu t(0,m_1).$ 
In particular, $t(1-p^e,m_1\mu)$ is bounded from above. 
However, as shown in Step 3 we have $dq_{p^e-1}\le t(1-p^e,m_1\mu)$, 
which implies that $t(1-p^e,m_1\mu)$ goes to infinity as $e$ goes to infinity. 
This is a contradiction.
\end{step}
\end{proof}
\begin{proof}[Proof of Theorem $\ref{thm:thmp2}$]  
The strategy of the proof is very similar to the one of Theorem \ref{thm:thmp} (1).
\setcounter{step}{0}
\begin{step}
In this step, we show that we may assume that $Y$ is regular.
Let $V$ be the regular locus of $Y$. Then we have $S\subseteq V$. 
As mentioned in Remark \ref{rem:wp} (2), 
the weak positivity of $-(K_Y+D)$ is equivalent to the weak positivity of $-(K_Y+D)|_V$.
Hence, replacing $f:X\to Y$ by $f_V:X_V\to V$, we may assume that $Y$ is regular. 
\end{step}
\begin{step}
In this step, we show that we may assume that $S\subseteq Y$ is open.
As shown in Step 1 of the proof of Theorem \ref{thm:thmp}, 
there exists an open subset $S'\subseteq Y$ containing $S$ such that 
the conditions (i)--(iv) hold for every $y\in S'$.
Let $B\subset X$ and $C\subset Y$ be, respectively, 
the supports of the cokernels of the natural morphisms  
$$f^*f_*\O_{X}(-L)\to\O_{X}(-L)\quad \textup{and}\quad H^0(X,f_*\O_X(-L))\otimes \O_Y\to f_*\O_X(-L). $$
Then $S\cap f(B)=S\cap C=\emptyset$ by the assumption. 
Hence we may replace $S$ by the open subset $S'\setminus(f(B)\cup C)$.
\end{step}
\begin{step}
In this step we reduce the proof to a numerical condition.
Let $A'$ be an $f$-ample and $f$-free Cartier divisor on $X$. 
Let $H$ be an ample and free Cartier divisor on $Y$. 
We take an integer $c>0$ and set $A:=A'+cf^*H$.
Applying Lemma \ref{lem:locfree} and Lemma \ref{lem:mult} (2),
there exists an integer $n'>0$ such that 
$f_*\O_X(N+nA)$ is locally free over $S$ and 
the natural morphism 
$$f_*\O_{X}(mA)\otimes f_*\O_{X}(N+nA)\to f_*\O_{X}(N+(m+n)A)$$ is surjective over $S$
for each $n\ge n'$, each $m>0$ and 
every Cartier divisor $N$ whose restriction $N_S$ to $X_S$ is $f_S$-nef. 
Note that since the statement is local on $Y$, $n'$ is independent of the choice of $c$.
Replacing $A$ by $n'A$, we may assume that $n'=1$. 
For simplicity, we set
\begin{align*}
D'&:=D+K_Y, \\
\G(l,m)&:=f_*\O_X(la^{-1}(aK_{X/Y}+E)+mA) \quad \textup{and} \\
t(l,m)&:=t_S(\G(l,m),D')
\end{align*}
for every integers $l,m$ with $a|l$. 
Note that since $-la^{-1}(aK_{X/Y}+E)_S$ is $f_S$-nef for each $l\ge0$ by the assumption,
we see that $\G(-l,m)$ is locally free over $S$ for each $l\ge0$ with $a|l$ and $m>0$.
Our goal is to prove that $t(-l,m)=+\infty$ for some integers $l,m\ge0$ with $a|l$. 
If this is shown, then the assertion follows from Lemma \ref{lem:ts} (4). 
Note that $\G(-l,m)$ is of positive rank, because of the $f_S$-freeness of 
$(-la^{-1}(aK_{X}+E)+mA)_S$.
In order to get a contradiction, we assume that 
$t(-l,m)\ne+\infty$ for every integers $l,m\ge0$ with $a|l$.  
\end{step}
\begin{step}
Set $d:=ab$.
Let $q_l$ and $r_l$ denote the quotient and the remainder 
of the division of an integer $l$ by $d$, respectively. 
In this step, we show that there exists an $l_0>0$ (independent of the choice of $c$) such that 
for each $l\ge l_0$ with $a|l$, 
$$dq_{l-l_0}+t(-r_{l-l_0}-l_0,1)\le t(-l,1).$$ 
By the projection formula, we have 
\begin{align*}
\left(\G(-d,0)\right)(-dD')\cong f_*\O_X(-b(aK_{X/Y}+E)-dD')=f_*\O_X(-b(aK_X+E+af^*D)).
\end{align*} 
The right hand side is globally generated over $S$ by the assumption, 
and hence so is the left hand side. This implies that 
\begin{align}
d\le t(-d,0)\label{align:d2}\tag{\ref{thm:thmp2}.1}
\end{align}
On the other hand, by the assumption, 
$-b(aK_{X/Y}+E)|_{X_S}=-L_S+ab{f_S}^*(D|_Y+K_Y)|_S$ is $f_S$-free. 
Applying Lemma \ref{lem:mult} (1), we have an integer $l_0>0$ such that 
for each $l\ge l_0$, the natural morphism 
\begin{align*}
f_*\O_X(-da^{-1}(aK_{X/Y}+E))\otimes& f_*\O_X(-la^{-1}(aK_{X/Y}+E)+A) \\
&\to f_*\O_X(-(d+l)a^{-1}(aK_{X/Y}+E)+A)
\end{align*}
is surjective over $S$. 
Note that since the statement is local on $Y$, $l_0$ is independent of the choice of $c$.
Then by an argument similar to Step 5 of the proof of Theorem \ref{thm:thmp}, we get
\begin{align}
t(-d,0)+t(-l,1)\le t(-d-l,1).\label{align:d3}\tag{\ref{thm:thmp2}.2}
\end{align}
Consequently, we obtain 
\begin{align*}
dq_{l-l_0}+t(-r_{l-l_0}-l_0,1) 
\overset{\textup{(\ref{align:d2})}}{\le} & q_{l-l_0}t(-d,0)+t(-r_{l-l_0}-l_0,l) \\
\overset{\textup{(\ref{align:d3})}}{\le} & t(-dq_{l-l_0}-r_{l-l_0}-l_0,1)=t(-l,1). 
\end{align*}
\end{step}
\begin{step}
We show that for each $m,n>0$ and $l\ge0$ with $a|l$,
\begin{align*} 
t(0,m)+t(-l,n)\le t(-l,m+n). 
\end{align*} 
As mentioned in the previous step, $-b(aK_{X/Y}+E)|_{X_S}$ is $f_S$-free, and so it is $f_S$-nef. 
By an argument similar to Step 5 of the proof of Theorem \ref{thm:thmp},
we obtain the claimed inequality.
\end{step}
\begin{step}
By the same argument as Step 6 of the proof of Theorem \ref{thm:thmp}, 
we see that there exists an $m_1>0$ such that for each $m\ge m_1$ and each $e\ge0$ with $a|(p^e-1)$,
\begin{align*}
t(1-p^e,mp^e)\le p^et(0,m).
\end{align*} 
\end{step}
\begin{step}
In this step we obtain a contradiction. 
Let $l_0$ be as in Step 4. 
Recall that $A:=A'+cf^*H$ and $l_0$ is independent of the choice of $c$.
Replacing $c$ by a larger integer, we may assume that 
for each integer $l$ with $0\le l\le d(l_0+1)$ and $a|l$, 
\begin{align*}
\G(l,1)=f_*\O_X(la^{-1}(aK_{X/Y}+E)+A'+cf^*H)\cong\left(f_*\O_X(la^{-1}(aK_{X/Y}+E)+A')\right)(cH)
\end{align*}
is globally generated, which implies that 
\begin{align}
0\le t(-l,1). \tag{\ref{thm:thmp2}.3}\label{align:-l}
\end{align}
Using the inequalities in the previous steps, we obtain that 
\begin{align*}
&(p^e-1)t(0,m_1)+dq_{p^e-1-l_0} \\ 
\overset{\textup{\hspace{0pt}(\ref{align:-l})\hspace{0pt}}}{\le} &(p^e-1)t(0,m_1)+(m_1-1)t(0,1)+t(-r_{p^e-1-l_0}-l_0,1)+dq_{p^e-1-l_0} \\
\overset{\textup{Step 4}}{\le} &(p^e-1)t(0,m_1)+(m_1-1)t(0,1)+t(1-p^e,1) \\
\overset{\textup{Step 5}}{\le} &(p^e-1)t(0,m_1)+t(1-p^e,m_1) 
\overset{\textup{Step 5}}{\le} t(1-p^e,m_1p^e) 
\overset{\textup{Step 6}}{\le} p^et(0,m_1).
\end{align*}
for each $e\ge0$ with $p^e-1\ge l_0$ and $a|(p^e-1)$.
Hence by the transposition we obtain $dq_{p^e-1-l_0}\le t(0,m_1)$. 
Note that since $0\le m_1 t(0,1)\le t(0,m_1)<+\infty$ by Step 5, we have $t(0,m_1)\in\R$. 
However, $dq_{p^e-1-l_0}$ goes to infinity as $e$ goes to infinity, which is a contradiction.
\end{step}
\end{proof}
In the remaining part of this section, we give some corollaries of 
the above theorems under the following situations: \\
\begin{notation}\label{notation:cors}
Let $f:X\to Y$ be a surjective morphism between regular projective varieties over an $F$-finite fields, 
and let $\Delta$ be an effective $\Q$-divisor on $X$ such that $a\Delta$ is integral. 
Let $D$ be a $\Q$-divisor on $Y$. 
Let $\ol\eta$ be the geometric generic point of $Y$. 
\end{notation}
\begin{cor} \label{cor:gl} 
With the notation as in \ref{notation:cors}, 
assume that $f$ is flat, that the support of $\Delta$ does not contain any component of any fiber, 
and that $(X_{\ol y},\Delta_{\ol y})$ is $F$-pure for every point $y\in Y$. 
\begin{itemize} 
\item[$(1)$] If $p\nmid a$ and if $-(K_X+\Delta+f^*D)$ is nef, then so is $-(K_Y+D)$. 
\item[$(2)$] If $-(K_X+\Delta+f^*D)$ is ample, then so is $-(K_Y+D)$.
\end{itemize}
\end{cor}
\begin{proof} 
This follows from Theorem \ref{thm:thmp} and Remark \ref{rem:wp} immediately. 
\end{proof}
%
The author learned the proof of Corollary \ref{cor:gen} (3) below from professor Yoshinori Gongyo.
\begin{cor} \label{cor:gen} 
With the notation as in \ref{notation:cors},
assume that $(X_{\ol\eta},\Delta_{\ol\eta})$ is $F$-pure. 
\begin{itemize} 
\item[$(1)$] If $p\nmid a$ and if $-(K_X+\Delta+f^*D)$ is nef, then $-(K_Y+D)$ is pseudo-effective.  
\item[$(2)$] If $-(K_X+\Delta+f^*D)$ is ample, then $-(K_Y+D)$ is big.  
\item[$(3)$] If $(X_{\ol\eta},\Delta_{\ol\eta})$ is strongly $F$-regular 
and if $-(K_X+\Delta+f^*D)$ is nef and big, then $-(K_Y+D)$ is big.  
\end{itemize}
\end{cor}
\begin{proof} 
By remark \ref{rem:wp}, (1) and (2) of the corollary follow from 
(1) and (2) of Theorem \ref{thm:thmp}, respectively. 
We prove (3). By Kodaira's lemma, there exists a $\Q$-divisor $\Delta'\ge\Delta$ on $X$ such that 
$-(K_X+\Delta'+f^*D)$ is ample and $(X_{\ol\eta},\Delta'_{\ol\eta})$ is again strongly $F$-regular. 
Hence (3) follows from (2).
\end{proof}
\begin{cor}\label{cor:cbf} 
With the notation as in \ref{notation:cors},
assume that $(X_{\ol\eta},\Delta_{\ol\eta})$ is $F$-pure. 
If $p\nmid a$ and if $K_X+\Delta$ is numerically equivalent to $f^*(K_Y+L)$ 
for a $\Q$-divisor $L$ on $Y$, then $L$ is pseudo-effective.
\end{cor}
\begin{proof} 
Set $D:=-(K_Y+L)$. Then $K_X+\Delta+f^*D$ is numerically trivial, and so it is nef.
Hence by Corollary \ref{cor:gen} (1), we obtain the assertion.
\end{proof}
\begin{rem}\label{rem:cbf1}
Assume that $K_X+\Delta\sim_{\Q}f^*(K_Y+L)$ for a $\Q$-divisor.
It is known that if $(X_{\ol\eta},\Delta_{\ol\eta})$ is globally $F$-split, 
then $L$ is $\Q$-linearly equivalent to an effective $\Q$-divisor on $Y$ 
(see \cite[Theorem B]{DS15} or \cite[Theorem 3.18]{Eji15}). 
However, $(X_{\ol\eta},\Delta_{\ol\eta})$ is not necessary globally $F$-split 
even if $X_{\ol\eta}$ is a smooth curve and $\Delta=0$. 
Incidentally, Chen and Zhang proved that relative canonical divisors of elliptic fibrations 
are $\Q$-linearly equivalent to an effective $\Q$-divisor on $X$ \cite[3.2]{CZ13b}. 
\end{rem}
\begin{rem}\label{rem:cbf2}
In the case when $\dim Y=1$, Corollary \ref{cor:cbf} follows from 
a result of Patakfalvi \cite[Theorem 1.6]{Pat14}.
\end{rem}
\begin{cor}\label{cor:fano} 
With the notation as in \ref{notation:cors},
assume that $f$ is flat and every geometric fiber is $F$-pure. 
\begin{itemize} 
\item[$(1)$] 
If $X$ is a Fano variety, that is, $-K_X$ is ample, then so is $Y$.  
\item[$(2)$] 
If the geometric generic fiber of $f$ is strongly $F$-regular and if $X$ is a weak Fano variety, 
that is, $-K_X$ is nef and big, then so is $Y$.  \end{itemize} 
\end{cor}
\begin{proof} 
This follows from Corollaries \ref{cor:gl} (2) and \ref{cor:gen} (3) by setting $D=\Delta=0$. 
\end{proof}
\begin{cor}\label{cor:notample} 
With the notation as in \ref{notation:cors}, assume that $Y$ is not a point.
\begin{itemize} 
\item[$(1)$] 
If $(X_{\ol\eta},\Delta_{\ol\eta})$ is $F$-pure, then $-(K_{X/Y}+\Delta)$ is not ample.
\item[$(2)$] 
If $(X_{\ol\eta},\Delta_{\ol\eta})$ is strongly $F$-regular, 
then $-(K_{X/Y}+\Delta)$ can not be both nef and big.
\end{itemize}
\end{cor}
\begin{proof} 
Set $D:=-K_Y$. Then $-(K_X+\Delta+f^*D)=-(K_{X/Y}+\Delta)$.
Since $-(K_Y+D)=0$ is not big, the assertions follows from Corollary \ref{cor:gen} (2) and (3).
\end{proof} 
\begin{cor} \label{cor:notbig} 
With the notation as in \ref{notation:cors}, 
assume that $\O_{X}(-m(K_{X}+\Delta))|_{X_{\eta}}$ is globally generated for an $m>0$ with $a|m$ 
and that $Y$ is not a point.
If $p\nmid a$ and if $(X_{\ol\eta},\Delta_{\ol\eta})$ is $F$-pure, 
then $f_*\O_X(-m(K_{X/Y}+\Delta))$ is not big.
\end{cor}
\begin{proof}
Set $\G(n):=f_*\O_X(n(K_{X/Y}+\Delta))$ for every $n\in \Z$ with $a|n$. 
Let $H$ be an ample and free divisor on $Y$.
We show that the bigness of $\G(-m)$ induces the pseudo-effectivity of $-H$,  
which contradicts to the assumption that $Y$ is not a point.
Recall that the pseudo-effectivity of $-H$ is equivalent to its weak positivity.
Since $\G(-m)$ is torsion free, there exists an open subset $V\subseteq Y$ 
such that $\G(-m)|_V$ is locally free and $\codim_Y(Y\setminus V)\ge 2$.
As explained in Remark \ref{rem:wp} (2), 
$\G(-m)$ (resp. $-H$) is weakly positive if and only if so is $\G(-m)|_V$ (resp. $-H|_V$).
Replacing $f:X\to Y$ by $f_V:X_V\to V$, we may assume that $\G(-m)$ is locally free.
(Here we lose the projectivity of $X$ and $Y$.) 
We assume that $\G(-m)$ is big. 
Then by Observation \ref{obs:wp} (1) and (2),
we see that $(S^l\G(-m))(-H)$ is generically globally generated for some $l>0$. 
By the global generation of $\O_{X}(-m(K_{X}+\Delta))|_{X_{\eta}}$  
and Lemma \ref{lem:mult} (1),
we have an $n_0>0$ such that the natural morphism 
$$(S^l\G(-m))^{\otimes n-n_0}\otimes \G(-lmn_0) \to \G(-lmn)$$ 
is generically surjective for every $n\ge n_0$. 
Let $n_1>0$ be an integer such that $(\G(-lmn_0))(n_1 H)$ is globally generated. 
Tensoring the above morphism with $\O_X((n_0+n_1-n)H)$, 
we have the generically surjective morphism 
\begin{align*}
\left((S^l\G(-m))(-H)\right)^{\otimes n-n_0}\otimes (\G(-lmn_0))(n_1 H) \to (\G(-lmn))((n_0+n_1-n)H).
\end{align*}
From this we see that $(\G(-lmn))((n_0+n_1-n)H)$ is generically globall generated. Since 
\begin{align*}
(\G(-lmn))((n_0+n_1-n)H) 
&\cong f_*\O_X(-lmn(K_{X/Y}+\Delta)+f^*(n_0+n_1-n)H) \\
&\cong f_*\O_X(-lmn(K_{X/Y}+\Delta+f^*H_n)) \\
&\cong f_*\O_X(-lmn(K_{X}+\Delta+f^*(H_n-K_Y))),
\end{align*}
where $H_n:=(n-n_0-n_1)(lmn)^{-1}H$,
we can apply Theorem \ref{thm:thmp2}, and then we get that $-(K_Y+(H_n-K_Y))=-H_n$ is 
weakly positive for every $n\ge n_0$. 
Hence $-H$ is weakly positive, which completes the proof.
\end{proof}
\begin{rem}\label{rem:counter-example} 
We cannot remove the assumption of $F$-purity of fibers 
in Corollaries \ref{cor:cbf} and \ref{cor:notbig}. 
Indeed, there is a quasi-elliptic fibration $g:S\to C$, that is, 
a surjective morphism from a smooth projective surface $S$ 
to a smooth projective curve $C$ whose general fibers are 
cuspidal curve of arithmetic genus one, 
such that $K_{S/C}\sim_{\Q}g^*L$ for a $\Q$-divisor $L$ on $C$ 
with $\deg L<0$ \cite{Ray78}\cite[Theorem~3.6]{Xie10}. 
Note that cuspidal singularities are not $F$-pure \cite{GW77}.  
\end{rem}
\section{Results in arbitrary characteristic}\label{section:thm0}
In this section we generalize some results in Section \ref{section:thmp} 
to arbitrary characteristic by the reduction to positive characteristic. 
In particular, we prove characteristic zero counterparts of 
Corollaries \ref{cor:notample} and \ref{cor:notbig} (Theorems \ref{thm:notample0} and \ref{thm:notbig0}). 
In the last of this section, we deal with morphisms which are special but not necessarily smooth, 
and show that images of Fano varieties are again Fano varieties. 
First we recall some classes of singularities in characteristic zero defined by reduction to positive characteristic. 
\begin{defn}\label{defn:Ftype}
Let $X$ be a normal variety over a field $k$ of characteristic zero, 
and $\Delta$ be an effective $\Q$-Weil divisor on $X$. 
Let $(X_R,\Delta_R)$ be a model of $(X,\Delta)$ over a finitely generated $\Z$-subalgebra $R$ of $k$. 
$(X,\Delta)$ is said to be of {\it dense $F$-pure type} (resp. {\it strongly $F$-regular type}) 
if there exists a dense (resp. dense open) subset $S\subseteq \Spec R$ such that 
$(X_{\mu},\Delta_{\mu})$ is {\it $F$-pure} (resp. strongly $F$-regular) for all closed points $\mu\in S$.
\end{defn}
\begin{rem}
The above definition can be generalized in an obvious way 
to the case where $X$ is a finite disjoint union of varieties over $k$.
\end{rem}
\begin{thm}[\textup{\cite[Corollary~3.4]{Tak04}}]\label{thm:strFreg} 
Let $X$ be a normal variety over a field of characteristic zero, 
and let $\Delta$ be an effective $\Q$-Weil divisor on $X$ such that $K_X+\Delta$ is $\Q$-Cartier. 
Then $(X,\Delta)$ is klt if and only if it is of strongly $F$-regular type.
\end{thm}
\begin{notation}\label{notation:zero}
Let $f:X\to Y$ be a surjective morphism between 
smooth projective varieties over an algebraically closed field of characteristic zero, 
and $\Delta$ be an effective $\Q$-divisor on $X$. 
\end{notation}
\begin{thm}\label{thm:notample0} 
With the notation as in \ref{notation:zero}, assume that $Y$ is not a point. 
If $(X_y,\Delta_y)$ is of dense $F$-pure type $($resp. klt$)$ for every general closed point $y\in Y$, 
then $-(K_{X/Y}+\Delta)$ is not ample $($resp. not nef and big$)$. 
In particular, $-K_{X/Y}$ is not nef and big.
\end{thm}
\begin{proof} 
Assume that $(X_y,\Delta_y)$ is of dense $F$-pure type for a general closed point $y\in Y$. 
Let $X_R$, $\Delta_R$, $Y_R$, $y_R$ and $f_R$ be respectively models of 
$X$, $\Delta$, $Y$, $y$ and $f$ over a finitely generated $\Z$-algebra $R$. 
We may assume that $(X_R)_{y_R}$ is a model of $X_y$ over $R$. 
Then there exists a dense subset $S\subseteq\Spec R$ such that 
$((X_y)_\mu,\Delta_\mu)$ is $F$-pure for every $\mu\in S$. 
Note that $(X_y)_\mu\cong (X_\mu)_{y_\mu}$ and $(\Delta_y)_\mu=(\Delta_\mu)_{y_\mu}$. 
Thus by Corollary \ref{cor:notample}, we see that $-(K_{X_\mu/Y_\mu}+\Delta_\mu)$ is not ample. 
This implies that $-(K_{X/Y}+\Delta)$ is not ample.
Next, we assume that $(X_y,\Delta_y)$ is klt for every general closed point $y\in Y$. 
If $-(K_{X/Y}+\Delta)$ is nef and big, then by Kodaira's lemma, 
there exists $\Delta'\ge\Delta$ such that $(X_y,\Delta'_y)$ is klt 
for a general closed point $y\in Y$ and $-(K_{X/Y}+\Delta')$ is ample. 
However, by Theorem \ref{thm:strFreg}, 
$(X_y,\Delta'_y)$ is of dense $F$-pure type, 
which contradicts to the above arguments. 
\end{proof}
\begin{thm}\label{thm:notbig0} 
With the notation as in \ref{notation:zero}, 
assume that $Y$ is not a point and that 
$(X_y,\Delta_y)$ is of dense $F$-pure type for a general closed point $y\in Y$. 
Let $\ol\eta$ be a geometric generic point of $Y$. 
If $\O_X(-m(K_{X/Y}+\Delta))|_{X_{\ol\eta}}$ is globally generated for some $m>0$ 
such that $m\Delta$ is integral, then $f_*\O_X(-m(K_{X/Y}+\Delta))$ is not big. 
\end{thm}
\begin{proof} 
Set $\G:=f_*\O_X(-m(K_{X/Y}+\Delta))$ and $r:={\rm rank}\;\G$. 
Since $y\in Y$ is general, $f$ is flat at every point in $f^{-1}(y)$ 
and $\dim H^0(X_y,-m(K_{X_y}+\Delta_y))=r.$ 
Let $X_R$, $\Delta_R$, $Y_R$, $y_R$ and $f_R$ be respectively models of $X$, $\Delta$, $Y$, $y$ and $f$. 
By replacing $R$ if necessary, we may assume that 
${f_R}_*\O_{X_R}(-m(K_{X_R/Y_R}+\Delta_R))$ and $(X_R)_{y_R}$ are respectively models of $\G$ and $X_y$,
and that for every $\mu\in\Spec R$, $\dim H^0((X_\mu)_{y_\mu},-m(K_{X_\mu}+\Delta_\mu)_{y_\mu})=r.$ 
Hence by \cite[Corollary~12.9]{Har77}, the natural morphism 
$$\G_\mu={f_R}_*\O_{X_R}(-m(K_{X_R/Y_R}+\Delta_R))|_{Y_\mu}
\to {f_\mu}_*\O_{X_\mu}(-m(K_{X_\mu/Y_\mu}+\Delta_\mu))$$ 
is surjective over $y_\mu$. 
Since ${f_\mu}_*\O_{X_\mu}(-m(K_{X_\mu/Y_\mu}+\Delta_\mu))$ is not big as shown in Corollary \ref{cor:notbig}, 
$\G_\mu$ is also not big. Thus the lemma below completes the proof.
\end{proof}
\begin{lem}\label{lem:bigred} 
Let $\G$ be a torsion free coherent sheaf on a smooth projective variety $Y$ 
over an algebraically closed field of characteristic zero. 
Let $Y_R$ and $\G_R$ be models of $Y$ and $\G$ respectively over a finitely generated $\Z$-algebra. 
If $\G$ is big, then there exists a dense open subset $S\subseteq\Spec R$ 
such that $\G_\mu$ is big for every $\mu\in S$.
\end{lem}
\begin{proof}
Let $H$ be an ample and free divisor on $Y$.
Replacing $Y$ by its appropriate open subset, we may assume that $\G$ is locally free. 
By Observation \ref{obs:wp} (1) and (2), we have an integer $c>0$ such that 
$(S^c\G)(-H)$ is generically globally generated.
In other words, there exists a generically surjective morphism 
$\varphi:\bigoplus\O_Y(H)\to S^c\G$.
Let $y\in Y$ be a general closed point. 
Then $\varphi$ is surjective in a neighborhood of $y$.
Let $\varphi_R$, $H_R$ and $y_R$ be models of $\varphi$, $H$ and $y$ over $R$, respectively. 
By replacing $R$ if necessary, we may assume that $\varphi_R$ is surjective over $y_R$. 
Thus for every closed point $\mu\in \Spec R$, the morphism 
$\varphi_\mu:\bigoplus\O_{X_\mu}(H_\mu)\to S^c\G_\mu$ obtained as the reduction of $\varphi_R$ 
is surjective over $y_\mu$. 
This implies that $\G_\mu$ is big, since $H_\mu$ is ample.
\end{proof}
Koll\'ar, Miyaoka and Mori \cite[Corollary 2.9]{KMM92} (cf. \cite[THEOREM 3]{Miy93}) proved that 
images of Fano varieties under smooth morphisms are again Fano varieties. 
The next theorem shows that the same statement holds when morphisms are not necessary smooth.
\begin{thm}\label{thm:toric}
Let $f:X\to Y$ be a surjective morphism between smooth projective varieties 
over an algebraically closed field $k$ of any characteristic $p\ge0$, 
and let $\Delta$ be an effective $\Q$-divisor on $X$ such that 
$a\Delta$ is integral for some $0<a\in\Z\setminus p\Z$. 
Assume that for every closed point $x\in X$, there exist a neighborhood $U\subseteq X$ 
$($resp. $V\subseteq Y$$)$ of $x$ $($resp. $f(x)$$)$ and a commutative diagram 
\begin{align*} 
\xymatrix@R=20pt@C=20pt{ 
U \ar[r]^(0.5){\alpha} \ar[d]_{(f_V)|_U} & 
\mathbb A^m \ar@{=}[r] \ar[d]^{\varphi} & 
\Spec k[u_1,\ldots,u_m]  \\ 
V \ar[r]^(0.5){\beta} & 
\mathbb A^n \ar@{=}[r] & 
\Spec k[v_1,\ldots,v_n]}
\end{align*}
whose horizontal morphisms are \'etale, that the morphism $\varphi$ is defined as 
\begin{align*}
\varphi(a_1,\ldots,a_m)=
\left(\prod_{0<i\le l_1}a_{i},\prod_{l_1<i\le l_2}a_i,\ldots,\prod_{l_{n-1}<i\le l_n} a_{i}\right)
\textup{ with $0<l_1<\cdots<l_n\le m$,} 
\end{align*}
and that 
$$ \Delta|_U=\alpha^*\left(\sum_{l_n<i\le m}d_i\mathrm{div}(u_i)\right)
\textup{ with $d_{l_n+1},\ldots,d_m\in\Z_{(p)}\cap[0,1]$.} $$ 
In this situation, if $-(K_X+\Delta+f^*D)$ is ample for some $\Q$-Cartier divisor $D$ on $Y$, 
then so is $-(K_Y+D)$.
\end{thm}
\begin{proof}
When the characteristic of $k$ is zero, 
one can easily check that the entire setting can be reduced to characteristic $p\gg0$. 
Thus we only need to consider the case when $p>0$. 
Let $x\in X$ be a closed point, and let $U$, $V$, $\alpha$, $\beta$ and $\varphi$ be as above. 
Set $y:=f(x)$, $\bm{a}:=\alpha(x)=(a_1,\ldots,a_m)\in \mathbb A^m$ 
and $\bm{b}:=\beta(y)=(b_1,\ldots,b_n)\in\mathbb A^n$. 
Set $u'_i:=u_i-a_i$ and $v'_j:=v_j-b_j$ for each $i$ and $j$.  Then 
$$ \varphi^*v'_j =\prod_{l_{j-1}<i\le l_j}(u'_i+a_i)-\prod_{l_{j-1}<i\le l_j}a_i $$ 
for $j=1,\ldots,n$, where $l_0:=0$. 
Set $g:=\prod_{l_n<i\le m}u_i.$ 
It is easy to check that the sequence $\varphi^*v'_1,\ldots,\varphi^*v'_n,g$ 
is $k[u_1,\ldots,u_m]$-regular. 
In particular, $Z:=Z(\varphi^*v'_1,\ldots,\varphi^*v'_n)\subseteq\mathbb A^m$ 
is equi-dimensional of dimension $m-n$.
Note that $Z$ is the fiber of $\varphi$ over $\bm b$.
Since \'etale morphisms are stable under base change, 
the morphism $\alpha_{\bm{b}}:U_{\bm b}\to Z$ obtained by 
the restriction of $\alpha$ to the fibers over $\bm b$ is again \'etale.
Replacing $V$ and $U$ if necessary, 
we may assume that $\beta^{-1}(\bm b)=\{y\}$. 
Then $U_y\cong U_{\bm b}$.
This implies that every closed fiber of $f$ is equi-dimensional, in particular $f$ is flat.
\begin{cl} 
$(X_y,\Delta_y)$ is $F$-pure for every closed point $y\in Y$. 
\end{cl}
If this claim holds, then the theorem follows from Corollary \ref{cor:gl}, 
because by the assumption the support of $\Delta$ does not contain any component of any fiber. 
Since $\sum_{l_n<i\le m}d_i\mathrm{div}(u_i)\le\mathrm{div}(g)$, 
it suffice to show that the pair $(Z, \mathrm{div}(g)|_Z)$ is $F$-pure around $\bm{a}$. 
Let $\mf m_{\bm a}$ be the maximal ideal of $\bm a$. 
Then it is easily seen that 
$$(\varphi^*v'_1\cdots\varphi^*v'_n)^{q-1}\cdot g^{q-1}\notin \mf m_{\bm{a}}^{[q]}$$ 
for every $q=p^e$. 
Thus by \cite[Corollary 2.7]{HW02}, the pair $(Z,\mathrm{div}(g)|_Z)$ is $F$-pure, 
which completes the proof. 
\end{proof}
\begin{eg} \label{eg:toric}
When $\Delta=0$, the assumptions of Theorem \ref{thm:toric} hold 
for flat toric morphisms with reduced closed fibers 
between smooth projective toric varieties over algebraically closed fields.
\end{eg}
\begin{eg}
\def\e{{\bm e}}
\def\v{{\bm v}}
Let $\{\e_1,\e_2,\e_3\}$ be the canonical basis of $\mathbb R^3$.
For integers $m,n\ge0$, we define $\v_{m,n}:=(1,m,n)\in\mathbb R^3$.
Let $\Sigma_{m,n}$ be the fan consisting of all the faces of the following cones:
\begin{align*}
& <\v_{m,n},\e_2,\e_2+\e_3>, <\v_{m,n},\e_2+\e_3,\e_3>, \\ 
&<\v_{m,n},-\e_2,\e_3>, <\v_{m,n},\e_2,-\e_3>, <\v_{m,n},-\e_2,-\e_3>, \\
& <-\e_1,\e_2,\e_2+\e_3>, <-\e_1,\e_2+\e_3,\e_3>, \\
&<-\e_1,-\e_2,\e_3>, <-\e_1,\e_2,-\e_3>, <-\e_1,-\e_2,-\e_3>.
\end{align*}
Let $X_{m,n}$ be the smooth toric 3-fold corresponding to the fan $\Sigma_{m,n}$
with respect to the lattice $\Z^3\subset\mathbb R^3$.
Then $X_{m,n}$ is a Fano variety if and only if $m,n\in\{0,1\}$.
The projection $\mathbb R^3\to\mathbb R^2:(x,y,z)\mapsto(x,y)$
induces a toric morphism $f:X_{m,n}\to Y_m$ from $X_{m,n}$ to the Hirzebruch surface 
$Y_m:=\mathbb P_{\mathbb P^1}(\O_{\mathbb P^1}\oplus\O_{\mathbb P^1}(m))$.
Set $\Delta=0$. Then one can check that 
$f$ satisfies the assumptions of Theorem \ref{thm:toric}, but is not smooth.
Hence by Theorem \ref{thm:toric}, we see that $Y_m$ is a Del Pezzo surface if $m=0,1$. 
In fact, it is well known that $Y_m$ is a Del Pezzo surface if and only if $m=0,1$.
\end{eg}
\section{Questions}\label{section:ques}
The first question is on a characteristic zero counterpart of Theorem \ref{thm:thmp}. 
Let $f:X\to Y$ be a flat surjective morphism between smooth projective varieties 
over an algebraically closed field of characteristic zero. 
Since $F$-pure singularities are considered to correspond to semi-log canonical singularities (cf. \cite{MS12}), 
we assume that every geometric fiber of $f$ is semi-log canonical. 
\begin{ques}\label{ques:zero} 
Suppose that $-K_X$ is ample (resp. nef, nef and big). 
Is $-K_Y$ also ample (resp. nef, nef and big)?
\end{ques}
The next question is on a generalization of Theorem \ref{thm:notample_intro}. 
Let $f:X\to Y$ be a generically smooth surjective morphism between smooth projective varieties 
over an algebraically closed field. Assume that $-K_{X/Y}$ is nef. 
By Theorem \ref{thm:notample_intro}, we see that 
if the Iitaka-Kodaira dimension $\kappa(X,-K_{X/Y})$ of $-K_{X/Y}$ is equal to $\dim X$, then $\dim Y=0$. 
In other words, if $\kappa(X,-K_{X/Y})$ is the maximum, then $\dim Y$ is the minimum. 
One can expect that this phenomenon is obtained as a consequence of a more general phenomenon. 
The following inequality is one of the possibilities.
\begin{ques}\label{ques:kod} 
Does the inequality $\kappa(X,-K_{X/Y})\le\kappa(X_{\ol\eta},-K_{X_{\ol\eta}})$ hold?
\end{ques}
If the inequality holds, then the above phenomenon can be viewed as its consequence. 
For instance, in characteristic zero, Question \ref{ques:kod} is known 
in the case where $X=\mb P(\E)$ for a vector bundle $\E$ of rank $r$ on $Y$ and $f$ is the natural projection. 
In this case, if $-K_{X/Y}$ is nef, then the numerical dimension of $-K_{X/Y}$ 
is equal to $r-1$ \cite[4.7.~Corollary]{Nak04} (cf. \cite{Yas11}). 
Thus $$\kappa(X,-K_{X/Y})\le r-1=\dim X_{\ol\eta}=\kappa(X_{\ol\eta},-K_{X_{\ol\eta}}).$$

\end{document}